\documentclass[final,1p,times,number]{elsarticle}
\usepackage[numbers]{natbib}
\usepackage[colorlinks]{hyperref}
\usepackage{ulem}
\usepackage{mathrsfs}
\usepackage{verbatim,amsfonts,amssymb, mathrsfs, amsmath,   amssymb, float,amsthm,enumitem,amsthm}

\usepackage{color}
\usepackage{epsfig}

\newtheorem{theorem}{Theorem}[section]
\newtheorem{lemma}[theorem]{Lemma}
\newtheorem{prop}[theorem]{Proposition}
\newtheorem{cor}[theorem]{Corollary}
\newtheorem{definition}[theorem]{Definition}

\newtheorem{example}[theorem]{Example}
\newtheorem{remark}[theorem]{Remark}

\usepackage{tikz}
\usetikzlibrary{arrows,shapes,positioning}
\usetikzlibrary{decorations.markings}
\tikzstyle arrowstyle=[scale=1]
\tikzstyle directed=[postaction={decorate,decoration={markings,mark=at position .65 with {\arrow[arrowstyle]{stealth}}}}]
\tikzstyle reverse directed=[postaction={decorate,decoration={markings,mark=at position .65 with {\arrowreversed[arrowstyle]{stealth};}}}]
\usepackage{subfigure}

\floatstyle{ruled}
\newfloat{algorithm}{t}{loa}
\floatname{algorithm}{Algorithm}

\def\ee{{\mathcal E}}
\def\rr{{\mathscr R}}
\def\aa{{\mathbb A}}
\def\I{{\mathbb I}}
\def\A{{\mathcal A}}
\def\P{{\mathcal P}}
\def\ff{{\mathcal F}}
\def\V{{\mathbb V}}
\def\C{{\mathcal C}}
\def\N{{\mathbb N}}
\def\Q{{\mathbb Q}}
\def\I{{\mathbb I}}
\def\sat{\hbox{\rm{sat}}}
\def\trdeg{\hbox{\rm{tr.deg}}}
\def\deg{\hbox{\rm{deg}}}
\def\Res{\hbox{\rm{Res}}}
\def\det{\hbox{\rm{det}}}
\def\ord{\hbox{\rm{ord}}}
\def\max{\hbox{\rm{max}}}
\def\H{\hbox{\rm{H}}}
\def\prem{\hbox{\rm{prem}}}
\def\S{\hbox{\rm{S}}}
\def\I{\hbox{\rm{I}}}
\def\gcld{\hbox{\rm gcld}}
\def\gcrd{\hbox{\rm gcrd}}

\begin{document}
 \title{Unirational Differential  Curves and Differential Rational Parametrizations}
\author{Lei Fu$^{\text{\,a, b}}$  and Wei Li$^{\text{\,a, b}}$ \smallskip\\  
$^{\text{a }}$KLMM,
   Academy of Mathematics and Systems Science,\\ 
       Chinese Academy of Sciences, Beijing 100190, China;\\ 
    $^{\text{b }}$University 
of Chinese Academy of Sciences, Beijing 100049, China\\
{fulei@amss.ac.cn, liwei@mmrc.iss.ac.cn} 
%\tnotetext[mytitlenote]{Partially supported by NSFC Grants (11971029, 11688101, 11671014).} 
 }

\begin{abstract}
In this paper, we study unirational differential curves and the corresponding differential rational parametrizations. 
We first investigate  basic properties of proper differential rational parametrizations for unirational differential curves. 
Then we show that the implicitization problem of proper linear differential rational parametric equations can be solved by means of differential resultants. 
{Furthermore,  for linear differential curves, we give an algorithm to determine whether an implicitly given linear differential curve is unirational and,  in the affirmative case, to compute a proper differential rational parametrization for the differential curve.}
\end{abstract}

 \maketitle
 
 \section{Introduction} \label{sec-intro} 

The study of unirational varieties and the corresponding rational parametrizations is a basic topic in computational algebraic geometry.  
The central problem in this field is to determine whether an algebraic variety is rationally parametrizable, and,  for unirational varieties, to give efficient algorithms to transform between the  implicit representations and  parametric representations. 
These problems have been fully explored for algebraic curves by Sendra, Winkler and P\'erez-D\'iaz in their classic book \cite{Winkler}  
using symbolic computation methods.  
They are also well-understood for algebraic surfaces \cite{Hartshorne, Sch98a}.

The differential implicitization and rational parametrization problems for differential algebraic varieties have similar importance, 
and thus the study of these problems is an active research field in differential algebraic geometry. 
The differential implicitization of differential rational parametric equations was first studied via the differential characteristic set method by Gao \cite{Gao2003}. In the special case when given linear differential polynomial parametric equations, it was treated via linear complete differential resultants \cite{Rueda, Rueda2011}.

However, as far as we know, there are still no general results on the differential parametrization problem, 
that is, the problem of deciding whether an implicitly given differential algebraic variety is differentially unirational, and of finding a differential rational parametrization in the affirmative case. 
The work of Feng and Gao, on finding rational general solutions for a univariate algebraic ODE $f(y)=0$, is the first step in solving the rational parametrization problem for differential varieties \cite{FengGao}. 
Their work  gives necessary and sufficient conditions for an ODE to have a rational general solution, and a polynomial algorithm to compute the rational general solution of a first order autonomous ODE if it exists.
Then the subsequent work by Winkler and his coauthors extended the method to study  rational general solutions for non-autonomous parametrizable first-order  ODEs \cite{Winkler2010, Winkler2011}, higher order ODEs \cite{Winkler2013} and even partial differential equations \cite{Winkler2018}.
While these are important contributions on the parametrization of zero-dimensional differential varieties in the one-dimensional space $\mathbb A^1$ \cite{Winkler2019},
 it seems that the rational parametrization problem for differential varieties of positive differential dimension has not been studied.

In this paper, we study unirational ordinary differential curves and the corresponding differential rational parametrizations. A  (plane) irreducible differential curve $\C$  is a one-dimensional irreducible differential variety in $\mathbb A^2$. The differential characteristic set method guarantees the unique existence of an irreducible differential polynomial $A(x,y)\in\ff\{x,y\}$ such that $\C$ is the general component of $A(x,y)=0$; 
thus this differential curve is often represented by $(\C, A)$. If $(\C, A)$ has a generic point of the form $\mathcal P(u)\in\ff\langle u\rangle^2$ with $u$ a differential parameter, it is called a unirational differential curve and $\mathcal P(u)$ is called a differential rational parametrization of $\C$.  $\mathcal P(u)$ is called proper if it defines a differential birational map between  $\mathbb A^1$ and $\C$. The differential L\"{u}roth theorem guarantees that  unirational differential curves always have proper differential rational parametrizations \cite{Gao2003}.

For unirational differential curves,  we first explore basic properties of proper differential rational parametrizations. 
In particular, Theorem \ref{th-order} gives the order property of properness and 
Theorem \ref{th-mobius} shows that proper parametrizations are unique up to M\"obius transformations.  
These results extend similar properties of proper parametrizations of algebraic curves to their differential counterparts.
For proper linear differential rational parametrizations, we give further properties and show differential resultants can be used to compute the corresponding implicit equations of proper linear differential rational parametric equations. 
This could be considered a generalization of Rueda-Sendra's work on implicitization of linear differential polynomial parametric equations via linear complete differential resultants \cite{Rueda}.

Concerning the rational parametrizability problem, it is well known that an algebraic curve is unirational if and only if its genus is equal to 0  \cite[Theorem 4.63]{Winkler}, so the determination of unirationality of algebraic curves can be reduced to the computation of the genus of the curve. 
Compared with the algebraic case, the rational parametrizability problem of differential curves is much more complicated to deal with. More precisely, the problem can be stated as follows: given an irreducible differential polynomial $A(x,y)$, decide whether the differential curve $(\C,A)$ is unirational, 
and if it is unirational, give efficient algorithms to compute a parametrization. 
In this paper, {we will start from the simplest nontrivial case by considering linear differential curves. 
And for  linear differential curves, we give an algorithm to determine whether the implicitly given differential curve is unirational,
 and, in the affirmative case, to  compute a proper linear differential polynomial parametrization for the unirational linear differential curve.}

This paper is organized as follows. In Section 2, we introduce some basic notions and notation in differential algebra. 
In Section 3,  we explore the basic properties of proper differential rational parametrizations for unirational {differential} curves.
 In Section 4,  further properties of proper linear differential rational parametrizations are given, and in particular,  
 the corresponding implicit equations can be computed via the method of differential resultants. 
 In Section 5, {we  deal with the problem of algorithmically deciding whether an implicitly given linear differential curve is unirational and computing a proper rational parametrization in the affirmative case.
In Section 6, we propose several problems for further study.}

\section{Basic notions and notation}
In the following, we will introduce the basic notions and notation to be used in this paper. 
For more details about differential algebra, please refer to \cite{Ritt},\cite{Kolchin1973},\cite{Wsit}.

Let $\ff$ be an ordinary differential field of characteristic 0 with derivation $\delta$. For example, $\ff=\Q(t)$ with $\delta=\frac{\mathrm d}{\mathrm dt}$.
 An element $c \in \ff$ such that $\delta(c)=0$ is called a constant of $\ff$. 
 The set of all constants of $\ff$ constitutes a differential subfield of $\ff$, called the field of constants of $\ff$ and denoted by $C_\ff$. 
% If $C_\ff=\ff$, $\ff$ is called a constant differential field.
 For an element $a$ in $\ff$, we use $a', a'', a^{(k)} $  to indicate the derivatives $\delta(a),  \delta^2(a), \delta^k(a)\,(k\geq3)$.

 Let $\mathcal G$ be a differential extension field of $\ff$.
 A subset $\Sigma\subset \mathcal G$ is said to be {\it  differentially dependent over $\ff$} if the set
$\Theta(\Sigma):=\{ \delta^{k}(a): a\in\Sigma, k\in\mathbb N\}$ is algebraically dependent over $\ff$.
Otherwise, $\Sigma$ is said to be a family of {\it  differential indeterminates over $\ff$}.
If $\Sigma=\{\alpha\}$, $\alpha$ is called {\it differentially algebraic} or {\it differentially transcendental} over $\ff$ respectively. 
Given $S\subset \mathcal G$, we denote  respectively by $\ff\{S\}=\ff[\Theta(S)]$ and $\ff\langle S\rangle=\ff\big(\Theta(S)\big)$ the smallest differential subring and differential subfield of $\mathcal G$ containing $\ff$ and $S$.

  Let $\ee$  be a fixed universal differential extension field of $\ff$ \cite[p. 134]{Kolchin1973}.
Let  $x,y, y_1,\ldots,y_n$ be a set of differential indeterminates over $\ee$.
Consider the differential polynomial ring $\ff\{y_1,\ldots,y_n\}=\ff[y_j^{(k)}:j=1,\ldots,n; k\in \N]$.
If $n=2$, we usually use the notation $\ff\{x,y\}$ instead.
A differential ideal in $\ff\{y_1,\ldots,y_n\}$ is an ordinary algebraic ideal  closed under $\delta$. 
A prime differential ideal is a differential ideal which is prime as an ordinary ideal. 
{For $\Sigma\subset\ff\{y_1,\ldots,y_n\}$, the differential ideal in $\ff\{y_1,\ldots,y_n\}$ generated by  $\Sigma$ is denoted by $[\Sigma]$.}
Let $f\in\ff\{y_1,\ldots,y_n\}$. For each $y_j$, the order of $f$ w.r.t. $y_j$  is defined to be the largest number $k$ such that $y_j^{(k)}$ appears effectively in $f$,  denoted by $\ord_{y_j}f$,  and in case $y_j$ and its derivatives don't appear in $f$,  we set $\ord_{y_j}f=-\infty$. 
{The \textit{order} of $f$ is defined to be $\max_{j=1}^n\{\ord_{y_j}f\}$, denoted by $\ord(f)$.  
\vskip2pt

Let $\mathbb A^n(\ee)$ denote the $n$-dimensional \textit{differential affine space} over $\mathcal E$. 
Let $\Sigma$ be a subset of differential polynomials in $\ff\{y_1,\ldots,y_n\}$. 
A point $\eta=(\eta_1,\ldots,\eta_n) \in \aa^{n}(\ee)$ is called a differential zero of $\Sigma$ if $f(\eta)=0$ for any $f \in \Sigma$. The set of differential zeros of $\Sigma$ is denoted by $\V(\Sigma)$, which is called a \textit{differential variety} defined over $\ff$. For a differential variety $V$, we denote $\mathbb I(V)$ to be the set of all differential polynomials in $\ff\{y_1,\ldots,y_n\}$ that vanish at every point of $V$. Clearly, $\mathbb I(V)$ is a radical differential ideal in $\ff\{y_1,\ldots,y_n\}$. A differential variety is said to be irreducible if it's not the union of two proper differential subvarieties. A point $\eta$ $\in \aa^{n}(\ee)$ is called a \textit{generic point} of a prime differential ideal $P$\,(or the corresponding variety $\V(P)$) if $\mathbb I(\eta)=P$. The \textit{differential dimension} of $P$ or $\V(P)$ is defined to be the differential transcendence degree of $\ff\langle \eta\rangle$ over $\ff$\cite[pp.105-109]{Kolchin1973}.

A ranking $\rr$ is a total ordering $<$ on  $\Theta(y):=\{y_j^{(k)}:j=1,\ldots,n; k\in \N\}$ such that for all $u,v \in \Theta(y)$, we have $u <\delta u$ and  $u < v \Rightarrow \delta u < \delta v$. 
By convention, $1<u$ for all $u \in \Theta(y)$. 
Let g be a differential polynomial in $\ff\{y_1,\ldots,y_n\}$ and $\rr$ a ranking endowed on it. The highest ranked derivative w.r.t. $\rr$ which appears in $g$ is called the \textit{leader} of $g$ and is denoted by $u_g$. Let $d$ be the degree of $g$ in $u_g$. Rewrite $g$ as a univariate polynomial in $u_g$, then, \begin{equation}g=I_du_g^{d}+I_{d-1}u_g^{d-1}+\cdots+I_0.\end{equation}
The leading coefficient $I_d$ is called the \textit{initial} of $g$ and denote it by $\I_g$. The partial derivative $\frac{\partial g}{\partial u_g}$ is called the \textit{separant} of $g$ and denoted by $\S_g$. The \textit{rank} of $g$ is defined to be $(u_g,d)$. Let $f$ and $g$ be differential polynomials and $(u_g,d)$ the rank of $g$. We say $f$ is \textit{partially reduced} w.r.t. $g$ if no proper derivative of $u_g$ appears in $f$, and $f$ is \textit{reduced} w.r.t. $g$ if $f$ is partially reduced w.r.t. $g$ and $\deg_{u_g}f <d$. Let $\A$ be a set of differential polynomials. $\A$ is said to be an \textit{autoreduced set} if each element of $\A$ is reduced w.r.t. every other one. Every autoreduced set is finite.

Let $\A=A_1,A_2,\cdots,A_t$ be an autoreduced set, and $f$ any differential polynomial.
 There exists a reduction algorithm, called Ritt-Kolchin's Remainder Algorithm, which reduces $f$ to a differential polynomial $r$ such that $r$ is reduced w.r.t. $\A$. More precisely, there exist $d_i,e_i \in \N$ such that 
 {\begin{equation}\prod_{i=1}^{t}\S_{A_i}^{d_i}\I_{A_i}^{e_i}\cdot f \equiv r, \text{mod } [\A].
 \end{equation}}
\noindent This $r$ is called the \textit{Ritt-Kolchin remainder} of $f$ w.r.t. $\A$.
Denote {$\H_\A=\prod_{i=1}^{t}\S_{A_i} \I_{A_i}$}. The \textit{saturation ideal} of $\A$ is defined as \begin{equation}\sat(\A)=[\A]:\H_\A^{\infty}=\{f\in \ff\{x,y\}\, \mid\, \exists\, m \in \N, \text{such that } \H_\A^{m}f \in [\A]\} \end{equation}
{Let $S\subset\ff\{y_1,\ldots,y_n\}$ be a differential polynomial set.}
An autoreduced set $\A$ contained in $S$ is said to be a {\it characteristic set } of $S$ if $S$ does not contain any nonzero element reduced w.r.t. $\A$. 
% All the characteristic sets of S have the same and minimal rank among all auto-reduced sets contained in S.   
A characteristic set $\A$ of a differential ideal $\mathcal I$ reduces to zero all elements of $\mathcal I$.
 If additionally $\mathcal I$ is prime, $\A$ reduces to zero only the elements of $\mathcal I$ and  $ \mathcal I=\sat(\A)$ (\cite[Lemma 2, p.167]{Kolchin1973}). The following result on Ritt's general component theorem will  often be used.
 \begin{lemma} \label{lm-satA} \cite[Theorem 11.2]{Wsit}
 Let $A\in\ff\{y_1,\ldots,y_n\}\backslash\ff$ be an irreducible differential polynomial and $S_A$ be the separant of $A$ under some ranking.
 Then $\sat(A)=[A]:S_A^\infty$ is a prime differential ideal, called the general component of $A$
 and $A$ is a  characteristic set of $\sat(A)$ under any ranking.
 In particular, if $B\in\sat(A)$ and $\ord(B)\leq\ord(A)$, then $B$ is divisible by $A$.
 \end{lemma}

We end this section by proving some technical results for later use.
Let $u\in\ee$ be differentially transcendental over  {$\ff$}.
Let $P, Q \in \ff\{u\}$ be  nonzero. 
The fraction $P/Q$ is called in reduced form if $\text{gcd}(P(u),Q(u))=1$.
Given  $R(u)\in\ff\langle u\rangle$ with $R(u)=\frac{P(u)}{Q(u)}$ in reduced form, 
the \textit{order} of $R(u)$ is defined to be $\max\{\ord_uP,\ord_uQ\}$, 
denoted by $\ord_u(R(u))$  {or simply by $\ord(R(u))$}.
Clearly, it is well-defined.
{
\begin{lemma}\label{lm-ord}
Let $P(u),Q(u) \in \ff\{u\}$ with $\gcd(P,Q)=1$ and $m=\ord(\frac{P(u)}{Q(u)})\geq0$. Then we have
\begin{enumerate}
\item[(1)]  
{For each $s\in\mathbb N_{>0}$}, $(P/Q)^{(s)}=\frac{P_s}{Q^{s+1}}$, where $P_s$ is a differential polynomial of order $m+s$ and linear in $u^{(m+s)}$; 
\item[(2)] $\trdeg_{\ff\langle \frac{P(u)}{Q(u)}\rangle}\ff\langle u \rangle=\ord(\frac{P(u)}{Q(u)})$.\end{enumerate}
\end{lemma}}
\proof
(1) Show it by induction on $s$.
 For $s=1$, $(P/Q)'=\frac{P'Q-PQ'}{Q^2}$ and $P_1=P'Q-PQ'=S_PQ\cdot u^{(\ord (P)+1)}-S_QP\cdot u^{(\ord(Q)+1)}+T$ with $\ord(T)\leq m$.
 If $\ord(P)\neq\ord(Q)$, clearly, $\text{rk}(P_1)=(u^{(m+1)},1)$.
Otherwise, since $\gcd(P,Q)=1$, $S_PQ-S_QP\neq0$ and $\text{rk}(P_1)=(u^{(m+1)},1)$  follows.
Suppose it holds for $s-1$. Then $(P/Q)^{(s)}=\big(\frac{P_{s-1}}{Q^{s}}\big)'=\frac{P'_{s-1}Q-sP_{s-1}Q'}{Q^{s+1}}$.
Let $P_s=P'_{s-1}Q-sP_{s-1}Q'$. By the induction hypothesis, $\text{rk}(P_s)=(u^{(m+s)},1)$.

(2) It is trivial for the case $m=0$. 
Consider the case when $m \geq 1$. 
Since $u, u',\cdots, u^{(m)}$ are algebraically dependent over $\ff\langle \frac{P(u)}{Q(u)} \rangle$, $\trdeg_{\ff\langle \frac{P(u)}{Q(u)} \rangle}\ff\langle u \rangle$ $=$ $\trdeg_{\ff\langle \frac{P(u)}{Q(u)} \rangle}\ff\langle \frac{P(u)}{Q(u)} \rangle ( u, u',\cdots, u^{(m-1)})$. If $u, u',\cdots, u^{(m-1)}$ are algebraically dependent over $\ff\langle \frac{P(u)}{Q(u)} \rangle$, there exists $s \in \N$ such that $u, u',\cdots, $ $u^{(m-1)}$ are algebraically dependent over $\ff\Big( \frac{P(u)}{Q(u)}, (\frac{P(u)}{Q(u)})', \cdots, (\frac{P(u)}{Q(u)})^{(s)}\Big)$. Thus 
{\small \begin{equation}
\begin{aligned}
 &\trdeg\,\ff\Big(u, u',\cdots, u^{(m-1)}, \frac{P(u)}{Q(u)}, (\frac{P(u)}{Q(u)})^{'}, \cdots, (\frac{P(u)}{Q(u)})^{(s)}\Big)\Big/\ff\\=&\trdeg\,\ff\Big(\frac{P(u)}{Q(u)}, (\frac{P(u)}{Q(u)})', \cdots, (\frac{P(u)}{Q(u)})^{(s)}\Big)\Big/\ff +\\&\trdeg\,\ff\Big(\frac{P(u)}{Q(u)}, (\frac{P(u)}{Q(u)})', \cdots, (\frac{P(u)}{Q(u)})^{(s)}\Big)\Big( u, u',\cdots, u^{(m-1)}\Big)\Big/{\ff\Big(\frac{P(u)}{Q(u)}, (\frac{P(u)}{Q(u)})^{'}, \cdots, (\frac{P(u)}{Q(u)})^{(s)}\Big)} \\ \leq& s+m. \nonumber
 \end{aligned}
\end{equation}
}
So $\trdeg\,\ff( u, u',\cdots, u^{(m-1)})(\frac{P(u)}{Q(u)}, (\frac{P(u)}{Q(u)})', \cdots, (\frac{P(u)}{Q(u)})^{(s)})\big/{\ff( u, u',\cdots, u^{(m-1)})} \leq s$,  contradicting the fact that $\frac{P(u)}{Q(u)}, (\frac{P(u)}{Q(u)})', \cdots, (\frac{P(u)}{Q(u)})^{(s)}$ are algebraically independent over $\ff( u, u',\cdots, u^{(m-1)})$. Thus, $\trdeg_{\ff\langle \frac{P(u)}{Q(u)} \rangle}\ff\langle u \rangle =m$.
\qed

By the additivity property of transcendence degrees,  {Lemma \ref{lm-ord} (2) }implies that the order is additive with respect to the composition of differential rational functions.
\begin{cor}\label{additive}
Let $R_1,R_2\in\ff\langle u\rangle\backslash\ff$. Then $\ord_u\big(R_2(R_1(u))\big)=\ord_uR_1+\ord_uR_2$ .
\end{cor}
\proof  
This follows by considering  $\ff\langle R_2(R_1(u))\rangle\subset\ff\langle R_1(u)\rangle\subset\ff\langle u\rangle$ and Lemma \ref{lm-ord} (2).
%% $\trdeg\,\ff\langle u\rangle/\ff\langle R_2(R_1(u))\rangle=\trdeg\,\ff\langle u\rangle/\ff\langle R_1(u)\rangle+\trdeg\,\ff\langle R_1(u)\rangle/\ff\langle R_2(R_1(u))\rangle.$
\qed

\vskip5pt
 
The following result is an exercise from \cite[p.159, Ex. 9]{Kolchin1973} which will be used in Section 3.
\begin{lemma}\label{lm-mobuis}
Let $t, u\in\mathcal E$ {be}  differentially transcendental elements over $\ff$ and $\ff\langle t\rangle=\ff\langle u\rangle$.
Then there exist $a,b,c,d\in\ff$ with $ad-bc\neq0$ such that $u=(at+b)/(ct+d)$.  
\end{lemma}
\proof Write $u=\frac{P(t)}{Q(t)}$ with $P,Q \in \ff\{y\}$ and $\gcd(P,Q)=1$. 
Observe that $P(y)-uQ(y)$ is irreducible in $\ff \langle u \rangle \{y\}$.
Let $\mathcal J=\sat(P-uQ)\subset \ff \langle u\rangle\{y\}$.
Fix a generic zero $s$ of $\mathcal J$, then we have $Q(s) \neq 0$. 
Indeed, if $Q(s)=0$, then $Q(y)$ is divisible by $P(y)-uQ(y)$ by Lemma \ref{lm-satA},  a contradiction. 
Thus, $u$ is differentially algebraic over $\ff \langle s\rangle$. 
Therefore, $s$ is differentially transcendental over $\ff$. 
So there exists a differential isomorphism $\phi:\ff \langle s\rangle \cong \ff \langle t \rangle$ with $\phi(s)=t$. 
Since $\phi(u)=u$, $\phi$ is a differential isomorphism over $\ff \langle u \rangle$.
Thus, $t$ is a generic zero of $\mathcal J$.  
Since $t$ is also a generic zero of $\sat(y-t)\subset \ff \langle u\rangle\{y\}$, $\mathcal J=\sat(y-t)$. 
By Lemma \ref{lm-satA}, $y-t$ is divisible by $P-uQ$ over $\ff\langle u\rangle$. 
So $P(y), Q(y) \in \ff[y]$ are both of degree at most 1. 
Thus, there exist $a,b,c,d\in\ff$ such that  $P(y)=ay+b$ and $Q(y)=cy+d$ with $ad-bc\neq0$.
 \qed

\section{Unirational differential curves and proper differential rational parametrizations}
In this section, we introduce the notions of unirational differential curves and proper diffferential rational parametrizations, and investigate the basic properties for  proper differential rational parametrizations.

\begin{definition}
A  {\rm differential curve} (over $\ff$) is a differential variety $\C\subset\mathbb A^n(\ee)$ (over $\ff$) which has  differential dimension 1. 
If additionally $\C$ is irreducible, $\C$ is called an irreducible differential curve.
A differential curve  $\C\subset \mathbb A^2$ is called a plane differential curve.
\end{definition}

Throughout the paper, we focus on the study of plane differential curves over the base differential field $\ff$,
 and so we always omit ``plane" and ``over $\ff$" for convenience.

If $\C\subset \mathbb A^2$ is an irreducible differential curve, then by the theory of differential characteristic sets, 
there exists a unique irreducible differential polynomial $A\in\ff\{x,y\}$ 
(up to an element in $\ff$) such that $\C$ is the general component of $A$  $\big($that is,  $\C=\V\big(\sat(A)\big)\big)$.
We call $\C$ {\it the differential curve defined by $A$}, and denote it by $(\C,A)$ for simplicity.

\begin{definition}(Unirational differential curves)\label{def1}
 Let $(\C,A)$ be an irreducible differential curve. 
We call $\C$ {\rm a unirational differential curve} if $\C$ has a generic point of the form 
\begin{equation} \label{eq-para}
\mathcal P(u)=\Big(\frac{P_1(u)}{Q_1(u)},\frac{P_2(u)}{Q_2(u)}\Big),
\end{equation}
where $u\in\ee$ is differentially transcendental over {$\ff \langle x,y \rangle$}, $P_i, Q_i \in \ff\{u\}$\footnote{Here we automatically have at least one $P_i/Q_i\in\ff\langle u\rangle\backslash\ff$.} 
and $\text{gcd}(P_i,Q_i)=1$ for $i=1,2$. 
And we call {\rm(\ref{eq-para})} a {\rm differential rational parametrization} of $\C$ or $A$.
 
\end{definition}

Recall that the differential L\"{u}roth theorem \cite{Ritt} tells us that any differential field $K$ with $\ff\subset K\subset \ff\langle u\rangle$ is generated by a single element. Combined with the differential L\"{u}roth theorem, we have the following equivalent definition for unirational differential curves in terms of the language of  differential fields.

\begin{prop}
An irreducible differential curve $\C$ is unirational if and only if the differential function field of $\C$, 
$\ff\langle \C \rangle=\text{\rm Frac}(\ff\{x,y\}/\mathbb I(\C))$, is differentially isomorphic to $\ff \langle u \rangle$ over $\ff$. 
\end{prop}
\proof   Let $\C$ be a unirational differential curve with a differential rational parametrization $\P(u)$.
By the differential L\"{u}roth theorem \cite{Ritt}, there exists $R(u) \in \ff \langle u \rangle \backslash \ff$ s.t. $\ff \langle \P(u) \rangle = \ff \langle R(u) \rangle$. Then it is easy to verify that the parametrization $\P(u)$ defines a differential isomorphism\begin{equation}\begin{aligned}\varphi: &\ff \langle \C \rangle \longrightarrow \ff \langle R(u) \rangle\\&\hspace{-0.12cm}f(x,y)\longmapsto f(\P(u)).\end{aligned}\end{equation} Since $R(u)$ is differentially transcendental over $\ff$,  $\ff \langle R(u) \rangle $ is differentially isomorphic to $\ff \langle u\rangle$. Therefore, $\ff\langle \C \rangle$ is differentially isomorphic to $\ff \langle u \rangle$. Conversely, let $\varphi: \ff\langle \C \rangle \rightarrow \ff \langle u \rangle$ be a differential isomorphism. Let $\P(u)= (\varphi(x), \varphi(y))$. Then $\P(u)\notin \ff^2$ and $\mathbb I(\P(u))=\mathbb I(\C)$. Thus, $\C$ is unirational with a differential rational parametrization $\P(u)$.
\qed.

\vskip3pt
For ease of notation, in this paper when we speak of a differential parametrization $\mathcal P(u)= (\frac{P_1}{Q_1},\frac{P_2}{Q_2})\in \ff\langle u\rangle^2$, we always assume each $P_i/Q_i$ is in reduced form, that is, $P_i,Q_i\in\ff\{u\}$ and $\gcd(P_i,Q_i)=1$. 
{And we define the order of $\mathcal P(u)$ to be $\max \{\ord_u(\frac{P_1}{Q_1}),\ord_u(\frac{P_2}{Q_2})\}$, denoted by $\ord(\mathcal P)$}.
The following result shows that differential parametrizations of a unirational differential curve satisfy certain order property.

\begin{prop} \label{prop-orderdifference}
Let $(\C,A)$ be a unirational differential curve  with $\ord_xA\geq0$ and $\ord_yA\geq0$. 
Suppose $\mathcal P(u)=(P_1/Q_1,P_2/Q_2)\in \ff\langle u\rangle^2$ is a differential rational parametrization of $\C$. 
Then $$\ord_xA+\ord_u(P_1/Q_1)=\ord_yA+\ord_u(P_2/Q_2).$$
In particular, $\ord_xA\leq \ord_u(P_2/Q_2)$ and $\ord_yA\leq \ord_u(P_1/Q_1)$.
\end{prop}
{\proof Denote $m_i=\ord_u(P_i/Q_i)$ for $i=1,2$. Then $m_i\geq0$.
Let $s_1=\ord_xA$ and $s_2=\ord_yA$. The fact $\mathbb I(\mathcal P(u))=\sat(A)\subset\ff\{x,y\}$ implies  that
$s_1$ and $s_2$ are respectively the minimal indices $\ell_1,\ell_2$ such that  
\begin{equation} \label{eq-orderdifference}
P_1/Q_1, (P_1/Q_1)',\ldots,(P_1/Q_1)^{(\ell_1)}, P_2/Q_2, (P_2/Q_2)',\ldots,(P_2/Q_2)^{(\ell_2)}
\end{equation} are algebraically dependent over $\ff$, or equivalently, in the module $(\Omega_{\ff\langle u\rangle}/\ff, d)$ of Kahler differentials, $d(P_1/Q_1), d(P_1/Q_1)',$ $\ldots,d(P_1/Q_1)^{(\ell_1)}, d(P_2/Q_2),d(P_2/Q_2)',\ldots,d(P_2/Q_2)^{(\ell_2)}$ are linearly dependent over $\ff\langle u\rangle$ \cite[p. 94]{Johnson1969}. 
By Lemma \ref{lm-ord}, each $(P_i/Q_i)^{(k)}$  is of order $m_i+k$, so $d(P_i/Q_i)^{(k)}$ is linear in $d(u^{(m_i+k)})$ for  {$k >0$}.
Thus, $m_1+s_1=m_2+s_2$ follows.
Note that when $\ell_1=m_2$ and $\ell_2=m_1$, 
the $m_1+m_2+2$ elements in (\ref{eq-orderdifference}) are contained in 
$\ff( u, u',\cdots, u^{(m_1+m_2)})$, and thus are algebraically dependent over $\ff$.
So $\ord_xA\leq m_2$ and $\ord_yA\leq m_1$.
%Since 
%\begin{center}$\ff \subset \ff\Big( \frac{P_1(u)}{Q_1(u)}, (\frac{P_1(u)}{Q_1(u)})', \cdots, (\frac{P_1(u)}{Q_1(u)})^{(n)},\frac{P_2(u)}{Q_2(u)}, (\frac{P_2(u)}{Q_2(u)})',\cdots, (\frac{P_2(u)}{Q_2(u)})^{(m)}\Big) \subset \ff( u, u',\cdots, u^{(m+n)}),$\end{center}we get 
%\begin{center}$\trdeg_\ff\ff\Big(\frac{P_1(u)}{Q_1(u)}, (\frac{P_1(u)}{Q_1(u)})', \cdots, (\frac{P_1(u)}{Q_1(u)})^{(n)},\frac{P_2(u)}{Q_2(u)}, (\frac{P_2(u)}{Q_2(u)})',\cdots, (\frac{P_2(u)}{Q_2(u)})^{(m)}\Big)\leq m+n+1.$\end{center}
%This implies $\frac{P_1(u)}{Q_1(u)}, (\frac{P_1(u)}{Q_1(u)})', \cdots, (\frac{P_1(u)}{Q_1(u)})^{(n)},\frac{P_2(u)}{Q_2(u)}, (\frac{P_2(u)}{Q_2(u)})^{'},\cdots, (\frac{P_2(u)}{Q_2(u)})^{(m)}$ are algebraically dependent over $\ff$. Thus $\ord_xA \leq n, \ord_yA \leq m$. 
\qed
}

\vskip5pt
Below we give some motivated examples and non-examples for unirational differential curves.

\begin{example} \label{ex-1}

\begin{itemize}
\item[(1)] Let $A=x'^2-4xy^2\in \mathbb Q(t)\{x,y\}$ with $\delta=\frac{\mathrm d}{\mathrm d t}$. 
Then $(\C,A)$ is a unirational differential curve with a differential rational parametrization $\mathcal P_1=(u^2,u')$.
Note that $\mathcal P_2=((u')^2,u'')$ is another  parametrization of $(\C,A)$
and $\ord(\mathcal P_1)<\ord(\mathcal P_2)$.

\item[(2)] Let $A=y'-x'\in \mathbb Q(t)\{x,y\}$. Then $(\C,A)$ is not unirational. 
{ Actually, if $A=B^{(k)}$ for some $B\in \ff\{x,y\}\backslash\ff$ and $k>0$, then $(\C,A)$ is not unirational. 
Or otherwise, there exists $\P(u)\in\ff\langle u\rangle^2$ such that $\mathbb I(\P(u))=\sat(A)$, 
which implies that $b=B(\P(u)) \in {\mathcal F\langle u\rangle}$ is differentially algebraic over $\ff$ and consequently,
$b\in\ff$. Thus,  $B(x,y)-b \in \sat(A)$.
By Lemma \ref{lm-satA}, $B(x,y)-b$ is divisible by $A$, a contradiction to $A=B^{(k)}$.}
% Let $A, B\in \ff\{x,y\}\backslash\ff$. If there exists a monic differential polynomial $T(z)\in \ff\{x,y\}\{z\}$ of positive order such that $A=T(B)$, then $A$ is not unirational.
\end{itemize}
\end{example}

From the above examples, we learn that not all differential curves are unirational and 
for a unirational differential curve, its differential rational parametrizations are not unique. 
{In fact, if $\mathcal P_1(u)$ is a differential rational parametrization of $(\C,A)$, 
then for any $R(u)\in  \ff\langle u\rangle\backslash  \ff$, $\mathcal P_2=\mathcal P_1\big(R(u) \big)$ is also a differential rational parametrization of $(\C,A)$, and thus $\C$ has infinitely many differential rational parametrizations.}
These facts lead to the following two natural problems:
\vskip2pt
\begin{center}\hskip.1truecm {\bf Problem 1.} Given $A\in\ff\{x,y\}$, decide whether the differential curve $(\C,A)$ is  unirational or not.

 \vskip2pt
\hskip-.26truecm {\bf Problem 2.} If $(\C,A)$ is unirational, find ``optimal" differential rational parametrizations for it w.r.t. some criteria, for instance,  having minimal order and degree .
 \end{center}

We first study Problem 2 in this section by introducing the notion of proper differential rational parametrizations and give the main basic properties, while leaving Problem 1 to be considered in  Section 5.

\begin{definition}
Let $\C$ be a unirational  differential curve with a differential rational parametrization $\mathcal P(u)=\left(\frac{P_1(u)}{Q_1(u)}, \frac{P_2(u)}{Q_2(u)}\right)$. 
The parametrization $\mathcal P(u)$ is said to be {\rm proper} if $\ff\langle\frac{P_1(u)}{Q_1(u)}, \frac{P_2(u)}{Q_2(u)}\rangle=\ff\langle u \rangle$.
\end{definition}

Equivalently, the notion of properness can be stated by means of differentially birational maps between $\C$ and $\mathbb A^1$. 
More precisely, let $(\C,A)$ be a unirational differential curve with a differential rational parametrization $\mathcal P(u)\in\ff\langle u\rangle^2$. 
This $\mathcal P(u)$  induces the differentially rational map
\[ \begin{array}{ccc} \mathcal P: \mathbb A^1 & --\hskip-.1truecm\to &\C\subset\mathbb A^2\\ 
\quad u &  {\longmapsto } &  \mathcal P(u). \end{array} \]
Then the differential parametrization $\mathcal P(u)$ is proper if and only if the map $\mathcal P$ is differentially birational, that is, $\mathcal P$ has an inverse differential rational map 
\[ \begin{array}{ccc} \mathcal U: \C & --\hskip-.1truecm\to& \mathbb A^1\\ 
\qquad (x,y) & \longmapsto &  U(x,y), \end{array} \]
where $U(x,y)\in\ff\langle x,y\rangle$ and the denominator of $U$ does not vanish identically on $\C$.
This $\mathcal U$ is called the inversion  of the proper differential rational parametrization $\mathcal P(u)$. 

In \cite{Gao2003}, Gao defined properness for differential rational  parametric equations (DPREs) and under his definition, $\left(\frac{P_1(u)}{Q_1(u)}, \frac{P_2(u)}{Q_2(u)}\right)$ is called proper if for a generic zero $(a_1,a_2)$ of $\C$, there exists a unique $\tau \in \ee$ such that $a_i=P_i(\tau)/Q_i(\tau)$. 
By \cite[Theorem 6.1]{Gao2003}, the equivalence of these  definitions can be easily seen.

At the same paper, Gao provided a method  \cite[Theorem 6.2]{Gao2003} to find a proper re-parameterization for any improper DPREs  based on a constructive proof of differential L\"uroth's theorem(\cite{Ritt}). 
Given an arbitrary rational parametrization $\mathcal P(u)$ of a unirational differential curve $\C$, Gao's method produces an algorithm to compute a proper differential rational parametrization for $\C$  from $\mathcal P(u)$, which in particular shows that each unirational differential curve has  a proper rational parametrization.

In the following we show that proper differential rational parametrizations possess essential properties of the unirational differential curves. 
We first give the order property of proper differential rational parametrizations. 
Recall that the order of a reduced differential rational function is equal to the maximum of the orders of its denominator and numerator.

\begin{theorem}  \label{th-order}
Let $(\C,A)$ be a unirational differential curve and $\big(\frac{P_1(u)}{Q_1(u)},\frac{P_2(u)}{Q_2(u)}\big)$ be a proper differential rational parametrization  of $\C$. Then we have 
\begin{equation} \label{eq-order}
\ord\Big( \frac{P_1(u)}{Q_1(u)}\Big)= \ord_yA,\,\,\, \ord\Big(\frac{P_2(u)}{Q_2(u)}\Big)= \ord_xA.
\end{equation}
\end{theorem}
\begin{proof}
For the special cases that  either $\frac{P_1(u)}{Q_1(u)}=a_1\in\ff$ or $\frac{P_2(u)}{Q_2(u)}=a_2\in\ff$,  
{by Lemma \ref{lm-mobuis} we have either 1) $A=x-a_1$ and   $\mathcal P(u)=(a_1,\frac{\alpha_1u+\beta_1}{\gamma_1u+\xi_1})$, 
or 2) $A=y-a_2$ and  
$\mathcal P(u)=(\frac{\alpha_2u+\beta_2}{\gamma_2u+\xi_2},a_2)$, for some $\alpha_i,\beta_i,\gamma_i,\xi_i \in \ff$ with $\alpha_i\xi_i-\beta_i\gamma_i \neq 0$, where (\ref{eq-order}) holds. }
So it suffices to consider the case that $m_i=\ord(\frac{P_i(u)}{Q_i(u)})\geq0$ for $i=1, 2$.

By \cite{Kolchin1947},  the relative order of $\sat(A)$ w.r.t. the parametric set $\{x\}$ is $$\ord_x\sat(A) = \trdeg_{\ff\langle \frac{P_1(u)}{Q_1(u)} \rangle}\ff\langle \frac{P_1(u)}{Q_1(u)}, \frac{P_2(u)}{Q_2(u)} \rangle =\trdeg_{\ff\langle \frac{P_1(u)}{Q_1(u)} \rangle}\ff\langle u \rangle.$$ 
By Lemma \ref{lm-ord}, $\trdeg_{\ff\langle \frac{P_1(u)}{Q_1(u)}\rangle}\ff\langle u \rangle =m_1$. 
Since $A$ is a characteristic set of $\sat(A)$ w.r.t. the elimination ranking: $x<y$, 
$\ord_yA= \ord_x(\sat(A))=m_1$ (\cite{charset}). 
By Proposition \ref{prop-orderdifference}, $\ord_xA-\ord_yA=m_2-m_1$, and thus $\ord_xA=m_2, \ord_yA=m_1$.
\end{proof}

\begin{remark} \label{rk-order} 
In the algebraic case, \cite[Theorem 4.21]{Winkler} characterizes the properness of a parametrization via the degree of the implicit equation of a unirational curve. That is, a parametrization $P(t)$ of a unirational curve $\mathcal V(f)$ is proper if and only if $\deg(P(t))=\max\{\deg_xf,\deg_yf\}$. However, in the differential case, we do not have such a characterization of properness via the orders of the differential curves
and the converse of Theorem \ref{th-order} is not valid.
For a non-example, let $A=y'-xy\in\ff\{x,y\}$. Clearly,  $\mathcal P(u)=(\frac{2u'}{u},u^{2})$ is a differential rational parametrization of $(\C, A)$ which satisfies (\ref{eq-order}). But $\ff\langle\mathcal P(u)\rangle=\ff\langle u^2 \rangle\neq \ff\langle u\rangle$, so $\mathcal P(u)$ is not proper.
\end{remark}

As a direct consequence of Theorem \ref{th-order}, we could show that  an algebraic curve  is unirational  in the algebraic sense if and only if it is unirational in the differential sense.

\begin{cor}
Let $(\ff,\delta)$ be a differential field which is algebraically closed. 
Let $A\in\ff[x,y]$ be an irreducible polynomial.
Then the differential curve $(\C,A)$ is unirational if and only if the genus of the algebraic curve defined by $A=0$ is $0$.
\end{cor}
\proof By \cite[Theorem 4.63]{Winkler}, an algebraic curve is rational if and only if its genus is 0.
So it suffices to show that if the differential curve $(\C,A)$ is unirational, 
then the algebraic curve defined by $A=0$ is rational in the algebraic sense. Indeed, if $(\C,A)$ is unirational with a proper differential rational parametrization $(\frac{P_1(u)}{Q_1(u)},\frac{P_2(u)}{Q_2(u)})$, by Theorem  \ref{th-order},  {$\ord(\frac{P_i(u)}{Q_i(u)})\leq0$}, which implies that the algebraic curve $A=0$ is rational.
\qed

\vskip5pt
Even if a unirational differential curve can have infinitely many proper differential rational parametrizations, 
 the following theorem shows that proper differential rational parametrizations enjoy some ``uniqueness" property up to  the so-called M\"{o}bius transformations.

\begin{theorem} \label{th-mobius}
If \,$\mathcal P_1(u), \mathcal P_2(u)$ are two proper differential rational parametrizations of $(\C,A)$, then there exist $a,b,c,d\in \ff$, s.t. $\mathcal P_2(u)= \mathcal P_1(\frac{au+b}{cu+d})$. Conversely, given any proper parametrization $\mathcal P(u)$ of $(\C,A)$, $\mathcal P(\frac{au+b}{cu+d})$ is also a proper parametrization of $(\C,A)$ for $ad-bc \neq 0$.
\end{theorem}
\begin{proof} Assume $\mathcal P_1(u)=\big( \frac{P_1(u)}{Q_1(u)}, \,\frac{P_2(u)}{Q_2(u)}\big), \mathcal P_2(u)=\big(\frac{P_3(u)}{Q_3(u)}, \frac{P_4(u)}{Q_4(u)}\big)$. 
Since $\ff\langle \frac{P_1(u)}{Q_1(u)}, \frac{P_2(u)}{Q_2(u)} \rangle=\ff\langle u \rangle$, 
there exist $M(x, y), N(x, y) \in \ff\{x,y\}$ s.t. $u= \frac{M(\mathcal P_1(u))}{N(\mathcal P_1(u))}$. Then,
\begin{center}$\frac{P_i\big(\frac {M(\mathcal P_1(u))}{N(\mathcal P_1(u))}\big)}{Q_i\big(\frac {M(\mathcal P_1(u))}{N(\mathcal P_1(u))}\big)}=\frac{P_i(u)}{Q_i(u)}, \,i = 1, 2.$\end{center}
Let $r(u)= \frac{M(\mathcal P_2(u))}{N(\mathcal P_2(u))}$. Since $\mathbb I\left(\frac{P_1(u)}{Q_1(u)}, \frac{P_2(u)}{Q_2(u)}\right)= \mathbb I\left(\frac{ P_3(u)}{Q_3(u)}, \frac{P_4(u)}{Q_4(u)}\right)$, we obtain $\frac{ P_3(u)}{Q_3(u)}=\frac{P_1(r(u))}{Q_1(r(u))}$, $\frac{ P_4(u)}{Q_4(u)}=\frac{P_2(r(u))}{Q_2(r(u))}$. Then $\ff\langle u \rangle =\ff\langle \frac{ P_3(u)}{Q_3(u)}, \frac{ P_4(u)}{Q_4(u)}\rangle \subset \ff \langle r(u) \rangle$, which implies $\ff\langle u \rangle = \ff\langle r(u) \rangle$. 
By Lemma \ref{lm-mobuis}, $r(u)=\frac{au+b}{cu+d}$ for some $a,b,c,d\in \ff$ with $ad-bc \neq 0$. 
The converse part is easy to check.
\end{proof}
{  \begin{remark}Let $(\C,A)$ be a unirational differential curve.
Theorem \ref{th-mobius} and its proof imply the following two facts about  proper differential rational parametrizations.
\begin{itemize}
\item[1).] Proper differential rational parametrizations of $(\C,A)$ are of the smallest order among all its rational parametrizations. 
%This could be easily shown by modifying the proof of Theorem \ref{th-mobius}.
Indeed, let $\mathcal P(u)$ be any given differential rational parametrization  of $(\C,A)$ and $\mathcal P_1(u)$ a proper one, then by the proof of Theorem \ref{th-mobius}, there  exists $r(u) \in \ff \langle u \rangle$ such that $\mathcal P(u)=\mathcal P_1(r(u))$. Therefore, by corollary \ref{additive}, $\ord(\mathcal P_1(u))\leq\ord (\mathcal P(u))$.

\item[2).]  Although the orders of proper differential rational parametrizations of  $(\C,A)$ are the same, 
their degrees can be distinct when $\ord(A)>0$.  
Take $A=y''-x$ for a simple example.
Clearly, $\mathcal P(u)=(u'', u)$ is a proper rational parametrization of  $(\C,A)$ of degree 1 and 
$\mathcal P(1/u)=(\frac{-uu''+2(u')^2}{u^3}, \frac{1}{u})$ is another proper parametrization of degree 3.
%\footnote{The degree of $\mathcal P(u)=(P_1/Q_1,P_2/Q_2)$ is defined as $\max_i\{\deg(P_i),\deg(Q_i)\}$.}  
It is interesting to estimate the  lowest degree of proper parametrizations   in terms of the numerical data of $A$.\end{itemize}
\end{remark}
}

We illustrate Theorem \ref{th-order} and Theorem \ref{th-mobius} by giving the following examples.
\begin{example}
\begin{itemize}
\item[(1)]
Let $A=y''x + (y')^2y - y'x' \in \Q \{x,y\}$.
Then $\mathcal P(u)=(uu'',u')$ is a proper parametrization of $(\C,A)$. 
Note that $\ord(uu'')=2=\ord_{y}A$ and $\ord(u')=1=\ord_{x}A$. 
\item[(2)]
Let $A=y'-x'-x\in\mathbb Q\{x,y\}$.  Then $\mathcal P_1=(u',u+u'), \mathcal P_2=(\frac{-u'}{u^{2}},\frac{u-u'}{u^{2}})$  are two proper parametrizations of $(\C,A)$.  Here, $\mathcal P_2(u)=\mathcal P_1(\frac{1}{u})$. Note that $\ord(\mathcal P_i)=\ord(A)$ and $\deg(\mathcal P_2)>\deg(\mathcal P_1)$.
\end{itemize}
\end{example}

\section{Proper linear differential rational parametrizations and the implicitization problem}

In this section, we first explore further properties for proper linear differential rational parametrizations 
and then study the corresponding implicitization problem by using  differential resultants.

\begin{definition} $\mathcal P(u)=\big(\frac{P_1(u)}{Q_1(u)}, \frac{P_2(u)}{Q_2(u)}\big)\in\ff\langle u\rangle^2\backslash\ff^2$ is called a {\rm linear differential rational parametrization} if for $i=1,2$, $P_i, Q_i\in \ff\{u\}$ are of degree at most 1 and  $\gcd(P_i,Q_i)=1$.
\end{definition}

Although the converse of Theorem \ref{th-order} is not valid in general as explained in Remark \ref{rk-order},
when restricted to linear differential rational parametrizations,
the next theorem shows that properness can be characterized via the orders of implicit equations of unirational differential curves.

\begin{theorem}\label{proper-rational}
Let $(\C,A)$ be a unirational differential curve which has a  linear differential rational parametrization $\mathcal P(u)=\big(\frac{P_1(u)}{Q_1(u)}, \frac{P_2(u)}{Q_2(u)}\big)$. Then,
$\mathcal P(u)$ is proper if and only if $$\ord\big( \frac{P_1(u)}{Q_1(u)}\big)= \ord_yA, \, \ord\big(\frac{P_2(u)}{Q_2(u)}\big)= \ord_xA.$$
\end{theorem}

\begin{proof}
Suppose $\ord\big( \frac{P_1(u)}{Q_1(u)}\big)= \ord_yA, \, \ord\big(\frac{P_2(u)}{Q_2(u)}\big)= \ord_xA$.
We need to show that $\mathcal P(u)$ is  proper.
Let $J=[Q_1(u)x-P_1(u), Q_2(u)y-P_2(u)]\colon(Q_1Q_2)^{\infty}$.
Then $J$ is a prime differential ideal in $\ff\{x,y,u\}$, and $\{Q_1(u)x-P_1(u), Q_2(u)y-P_2(u)\}$ is its characteristic set w.r.t. the elimination ranking $u<x<y$ \cite[p.107]{Ritt}. 
Now we compute a characteristic set $B_1(x,y), B_2(x,y,u)$ of $J$ w.r.t. the elimination ranking $x<y<u$. 
Since $P_i, Q_i$ are of degree at most 1, by the zero-decomposition theorem, $B_2(x,y,u)$ is a linear differential polynomial in $u$.

If  $\mathcal P(u)$ is not proper, then $s=\ord_uB_2(x,y,u)\geq1$. 
Rewrite $B_2$ in the form $B_2(x,y,u)=I_s(x,y)u^{(s)}+\cdots+I_0(x,y)u+I(x,y)$, where $I_i(x,y), I(x,y) \in \ff\{x,y\}$. 
Let $\bar B(z)=\frac{B_2(\frac{P_1(u)}{Q_1(u)},\frac{P_2(u)}{Q_2(u)},z)}{I_s(\frac{P_1(u)}{Q_1(u)},\frac{P_2(u)}{Q_2(u)})}\in \ff\langle \frac{P_1(u)}{Q_1(u)},\frac{P_2(u)}{Q_2(u)}\rangle\{u\}$, 
then there exists a coefficient $k_0$ of $\bar B(z)$ such that $\ord_uk_0\geq1$.  
By the proof of the differential L\"uroth theorem (\cite{Kolchin1947}), we obtain $\ff \langle \frac{P_1(u)}{Q_1(u)},\frac{P_2(u)}{Q_2(u)} \rangle =\ff \langle k_0 \rangle$. Thus, there exist $P_3(u), Q_3(u),P_4(u), Q_4(u) \in \ff\{u\}$ such that $\frac{P_3(k_0)}{Q_3(k_0)}=\frac{P_1(u)}{Q_1(u)}, \frac{P_4(k_0)}{Q_4(k_0)}=\frac{P_2(u)}{Q_2(u)}$, and we have $(\frac{P_3(u)}{Q_3(u)},\frac{P_4(u)}{Q_4(u)})$ is a proper parametrization of $(\C,A)$ with $\ord(\frac{P_3(u)}{Q_3(u)}) \leq \ord_yA-1, \ord(\frac{P_4(u)}{Q_4(u)}) \leq \ord_xA-1$, 
which contradicts theorem \ref{th-order}. 
Thus, $\mathcal P(u)$ should be proper.
Combined with Theorem \ref{th-order}, $\mathcal P(u)$ is  proper if and only if $\ord\big( \frac{P_1(u)}{Q_1(u)}\big)= \ord_yA, \, \ord\big(\frac{P_2(u)}{Q_2(u)}\big)= \ord_xA.$
\end{proof}

\vskip3pt
In \cite{Rueda}, the implicitization problem for linear differential polynomial parametric  equations was studied via linear complete differential resultants. 
In the following, we present results on implicitization for  linear differential rational parametrizations with the method of differential resultants.
Before that, we need a technical result.

Recall that the wronskian determinant of $\xi_1,\ldots,\xi_n$ is 
\[ \text{wr}(\xi_1,\ldots,\xi_n)=\left|\begin{matrix} 
\xi_1&\delta(\xi_1)&\cdots&\delta^{n-1}(\xi_1)\\
\xi_2&\delta(\xi_2)&\cdots&\delta^{n-1}(\xi_2)\\
\cdots &\cdots&\cdots&\cdots\\
\xi_n&\delta(\xi_n)&\cdots&\delta^{n-1}(\xi_n)
\end{matrix}\right|.\]
It is well-known that $\text{wr}(\xi_1,\ldots,\xi_n)=0$ gives a necessary and sufficient condition that $\xi_1,\ldots,\xi_n$ are linearly dependent over constants.

\begin{lemma} \label{lm-rightdivisor}
Let $L=\delta^n+a_{n-1}\delta^{n-1}+\cdots+a_0\in \ff[\delta]$.
Suppose $L_1=\delta^{n_1}+b_{n_1-1}\delta^{n_1-1}+\cdots+b_0\in  \mathcal G[\delta]$ is a right divisor of $L$ over some differential extension field  $\mathcal G$ of $\ff$.
Then  all the $b_i$ belong to a finite differential algebraic extension field of $\ff$.
In particular, $\text{\rm tr.deg}\,\ff\langle b_0,\ldots,b_{n_1-1}\rangle/\ff<\infty.$
\end{lemma}
\proof 
$\text{Sol}(L)=\{y\in \mathcal E \mid L(y)=0\}$ is a linear space of dimension $n$ over $C_\mathcal E$, the field of constants of $\mathcal E$.
Let $\xi_1,\ldots,\xi_n$ be a basis of $\text{Sol}(L)$. Then the $a_{n-1}, \ldots,a_{1},a_0$ satisfy the following linear equations
$$\delta^n(\xi_i)+a_{n-1}\delta^{n-1}(\xi_i)+\cdots+a_0\xi_i=0\,(i=1,\ldots,n).$$
So we have   \[ \left(\begin{matrix} 
\xi_1&\delta(\xi_1)&\cdots&\delta^{n-1}(\xi_1)\\
\xi_2&\delta(\xi_2)&\cdots&\delta^{n-1}(\xi_2)\\
\cdots &\cdots&\cdots&\cdots\\
\xi_n&\delta(\xi_n)&\cdots&\delta^{n-1}(\xi_n)
\end{matrix}\right)\left(\begin{matrix} 
a_0\\
a_1\\
\vdots\\
a_{n-1}
\end{matrix}\right)
=\left(\begin{matrix} 
-\delta^{n}(\xi_1)\\
-\delta^{n}(\xi_2)\\
\vdots\\
-\delta^{n}(\xi_n)
\end{matrix}\right). \]
Since $\xi_1,\ldots,\xi_n$ are linearly independent over $C_\mathcal{E}$, the wronskian determinant $\text{wr}(\xi_1,\ldots,\xi_n)$ is nonzero.
Thus, $a_i=\frac{\text{wr}_i(\xi_1,\ldots,\xi_n)}{\text{wr}(\xi_1,\ldots,\xi_n)}\in \ff$, where $\text{wr}_i(\xi_1,\ldots,\xi_n)$  is obtained from $\text{wr}(\xi_1,\ldots,\xi_n)$ by replacing its  $(i+1)$-th column by $(-\delta^{n}(\xi_1)\,\, -\delta^{n}(\xi_2)\,\,\cdots\,\,-\delta^{n}(\xi_n))^{\text{T}}$. 

Since $L_1$ is a right divisor of $L$, the solution space $\text{Sol}(L_1)$ $(\subset\text{Sol}(L))$ of $L_1$ in $\mathcal E$ is of dimension $n_1$ 
and there exist $c_{ij}\in C_{\mathcal E}\,(i=1,\ldots,n_1;j=1,\ldots,n)$ such that 
$$\eta_i=c_{i1}\xi_1+c_{i2}\xi_2+\cdots+c_{in}\xi_n,\,i=1,\ldots,n_1$$ is a  basis of   $\text{Sol}(L_1)$. 
Similarly, we can recover the coefficients $b_j$ of $L_1$ from the $\eta_i$'s following the above steps.
Thus, $b_j\in \ff(c_{ij})\langle\xi_1,\ldots,\xi_n\rangle$, which is a finitely generated differential algebraic extension field of $\ff$. \qed

Differential resultant for two univariate differential polynomials  was first introduced by Ritt \cite{Ritt1932}.  
 Carr\`a-Ferro then proposed to use algebraic Macaulay resultants to compute differential resultants for $n+1$ differential polynomials in $n$ differential variables \cite{Ferro1997, Ferro}, which is incomplete in that under her method, even the differential resultant of two generic nonlinear univariate  differential polynomials is identically zero. 
 The first rigorous definition of the differential resultant for $n+1$ differential polynomials in $n$ differential variables was given in \cite{Gao2013}. 
Although Ferro's matrix formulae do not work for the general case, 
these definitions for the differential resultant of  two linear univariate differential polynomials are equivalent. 
Now we recall the definition of differential resultants for two linear univariate differential polynomials
 via the matrix formulae.
 
\begin{definition}\label{dresultant}
 Let $f_1, f_2 \in \mathbb D\{u\}$ be linear differential polynomials of order $m_1, m_2$ over a differential domain $\mathbb D$. 
Let $PS=\big\{f_1^{(m_2)}, f_1^{(m_2-1)},\cdots,f_1,f_2^{(m_1)},f_2^{(m_1-1)},\cdots,f_2\big\}$ and set  $L=m_1+m_2+2$.
Let $M$ be the $L\times L$ matrix whose $k$-th row is the coefficient vector of the $k$-th polynomial in $PS$ w.r.t.  $u^{(m_1+m_2)}>u^{(m_1+m_2-1)}>\cdots>u>1$ (i.e., the resultant matrix of $PS$ w.r.t. $u^{(j)}, j\leq m_1+m_2)$. 
This $M$ is called the {\rm differential resultant matrix} of $f_1$ and $f_2$ w.r.t. $u$, and $\det(M)$ is defined to be the {\rm differential resultant} of $f_1, f_2$,
denoted by $\delta\text{-}\Res(f_1,f_2)$.
\end{definition}

The following result shows that the differential resultant can be applied to compute the implicit equation for proper linear differential rational   parametric  equations.
\begin{theorem} \label{d-resultant}
Let  $\mathcal{P}(u)=(\frac{P_1(u)}{Q_1(u)},\frac{P_2(u)}{Q_2(u)})$ be a linear differential rational parametrization with 
$m_i=\ord(\frac{P_i}{Q_i}) \geq 0$ for $i= 1,2$. 
If $\mathcal{P}(u)$ is proper, then the differential resultant $$R(x,y):=\delta\text{-}\Res_u\big(xQ_1(u)-P_1(u),yQ_2(u)-P_2(u)\big) \neq 0.$$
Furthermore, $\ord_xR=m_2, \ord_yR=m_1$ and $R$ is  linear in $x^{(m_2)}$ and $y^{(m_1)}$.
\end{theorem}
\proof
Let $f_1=xQ_1(u)-P_1(u),  f_2=yQ_2(u)-P_2(u)\in \ff\{u,x,y\}$. 
Denote $m=m_1+m_2+2$.
Let $M\in \ff\{ x,y\}^{m\times m}$ be the differential resultant matrix of $f_1, f_2$ w.r.t. $u$.
Then $R:=\delta\text{-}\Res(f_1,f_2)=\det(M)\in  \ff\{ x,y\}$.
To show $R\neq0$, it suffices to prove that $\text{coeff}(\det(M), x^{(m_2)})\neq 0$.

Note the fact that only $f^{(m_2)}$ effectively involves $x^{(m_2)}$ and $f^{(m_2)}$ is linear in $x^{(m_2)}$ with coefficient $Q_1(u)$.
So $\text{coeff}(\det(M), x^{(m_2)})=\det(M_1)$, where $M_1\in \ff\{x,y\}^{m\times m}$ be the resultant matrix of $Q_1,f_1^{(m_2-1)},\ldots,f_1'$, $f_1,f_2^{(m_1)},\ldots,$ $f_2',f_2$ w.r.t. the variables $u^{(m_1+m_2)}, u^{(m_1+m_2-1)},\ldots,u',u$. 
For $j=1,\ldots, {m-1}$, multiply the $j$-th column of $M_1$ by $u^{({m-1-j})}$ 
 and add it to the last column, then compute $ \det(M_1)$ by the last column.
So there exist $a_i,b_j,a \in \ff\{x,y\}$ such that 
 \begin{equation}
 \label{eq-det}
 \det(M_1)=a(x,y)Q_1(u)+\sum_{i=0}^{m_2-1}a_i(x,y)f_1^{(i)}+\sum_{j=0}^{m_1}b_j(x,y)f_2^{(j)}.
 \end{equation}  
Clearly, $a(x,y)=p(y) \cdot \det(M_2) \cdot (-1)^{m_1+1}$, where $p(y)=\text{coeff}(f_2,u^{(m_2)})\neq 0$ and $M_2$ is the submatrix  obtained from $M_1$ by removing the $1$-th, the $(m_2+2)$-th rows, and the $1$-th, the {$m$-th} columns. 
 We claim that $\det(M_2)\neq0$.  
 
If $m_i=0$ for some $i$, then $\det(M_2)=(\text{coeff}(f_i,u))^{m-2} \neq 0$. 
%If $m_2=0, \det(M_2)=(\text{coeff}(f_2,u))^{m_1} \neq 0$. 
Now suppose $m_1,m_2 >0$.
Assume $\mathcal{P}(u)=(\frac{L_{11}(u)+a_{11}}{L_{12}(u)+a_{12}},\frac{L_{21}(u)+a_{21}}{L_{22}(u)+a_{22}})$ where $L_{ij}\in \ff[\delta]$, and for each $i$, $L_{i1}$ and $L_{i2}$ are not both equal to zero.
Clearly, $$\det(M_2)=\delta\text{-}\Res^{h}(xL_{12}(u)-L_{11}(u),yL_{22}(u)-L_{21}(u)).$$
By \cite[Theorem 2]{chardin}, $\delta$-$\Res^{h}(xL_{12}(u)-L_{11}(u),yL_{22}(u)-L_{21}(u))\neq0$ if and only if $\text{gcrd}(xL_{12}-L_{11},yL_{22}-L_{21})=1$ (over $\ff\langle x,y\rangle$).  
We now show $\text{gcrd}(xL_{12}-L_{11},yL_{22}-L_{21})=1$.
Suppose the contrary, that is, $\text{gcrd}(xL_{12}-L_{11},yL_{22}-L_{21})=D(\delta)$ which is monic of degree greater than 0.
By Lemma \ref{lm-rightdivisor}, $D(\delta)$ can not effectively involve $x$ or $y$. For if not, suppose $D(\delta)$ effectively involves $x$, which contradicts the fact that the coefficients of a monic right divisor of $yL_{22}-L_{21}$ have finite transcendence degree over $\ff\langle y\rangle$ 
($x$ is differentially transcendental over $\ff\langle y\rangle$).
So $D(\delta)\in \ff[\delta]$. 
As a consequence, $\ff\langle\mathcal{P}(u)\rangle\subseteq \ff\langle D(u)\rangle\subsetneqq \ff\langle u\rangle$,    contradicting the hypothesis that $\mathcal{P}(u)$ is proper. 
Thus, $\text{gcrd}(xL_{12}-L_{11},yL_{22}-L_{21})=1$ and $\det(M_2)\neq0$ follows.

Since $\det(M_2)\neq0$, $a(x,y)=(-1)^{m_1+1}p(y)\cdot \det(M_2) \neq 0$.
Let $A(x,y)\in\ff\{x,y\}$ be an irreducible differential polynomial such that $\mathcal P(u)$ is the general component of $A$.
Since $\mathcal P(u)$ is proper, by Theorem \ref{th-order},  
$\ord_xA=m_2$ and $\ord_yA=m_1$.
The fact $\ord_x\det(M_2)< m_2$ implies that $\ord_xa(x,y)<m_2$. 
So $a(\mathcal P(u)) \neq 0$. 
By (\ref{eq-det}), det$(M_1)\big|_{(x,y)=\mathcal P(u)} \neq 0$ and thus $\text{coeff}(R(x,y), x^{(m_2)})=\det(M_1)\neq0$.
So $R\neq0$, $\ord_xR=m_2$ and $R$ is linear in $x^{(m_2)}$.
Since $R\in [f_1,f_2]{\,\cap\, \ff\{x,y\}} \subset \mathbb{I}(\mathcal P(u))=\sat(A)$,
$\ord_yR\geq \ord_yA=m_1$. Since $\ord_yR\leq m_1$,  $\ord_yR=m_1$ and $R$ is linear in $y^{(m_1)}$.
This completes the proof.  
\qed

\vskip5pt
By Theorem \ref{d-resultant},  the implicitization of a proper linear differential rational parametrization can be reduced to the computation of the corresponding differential resultant.
\begin{cor} \label{cor-resultant}
Let $(\C,A)$ be a unirational differential curve with a  linear differential rational parametrization $\mathcal P(u)=\big(\frac{P_1(u)}{Q_1(u)}, \frac{P_2(u)}{Q_2(u)}\big)$ {with $m_i=\ord(\frac{P_i}{Q_i}) \geq 0$ for $i= 1,2$}. 
Suppose $\mathcal P(u)$ is proper. 
Then  $A$ is the main irreducible factor of $R=\delta\text{-}\Res\big(xQ_1(u)-P_1(u),yQ_2(u)-P_2(u)\big)$. That is,
if $R=AB$, then $\ord(B)<\ord(R)$. In particular,  $$\deg_{x^{(m_2)}}A=1  \text{ and } \deg_{y^{(m_1)}}A=1.$$
\end{cor}

Another direct consequence of Theorem \ref{d-resultant} gives a necessary condition for a unirational differential curve to possess a proper linear differential rational parametrization.

\begin{cor}
Suppose $A\in \ff\{x,y\}$ defines a unirational differential curve.
A necessary condition such that $(\C,A)$ has a proper linear differential rational parametrization is that $A$ is quasi-linear under any ranking.
\end{cor}
%\textcolor{teal}{
%\begin{cor}
%Let $(\C,A)$ be a unirational differential curve with a  linear differential rational parametrization $\mathcal{P}(u)=(\frac{L_{11}(u)+a_{11}}{L_{12}(u)+a_{12}},\frac{L_{21}(u)+a_{21}}{L_{22}(u)+a_{22}})$ where $L_{ij}\in \ff[\delta]$, and for each $i$, $L_{i1}$ and $L_{i2}$ are not both equal to zero, a necessary condition for $\mathcal P(u)$ being proper is that $\delta\text{-}\Res^{h}(xL_{12}(u)-L_{11}(u),yL_{22}(u)-L_{21}(u)) \neq 0$.\end{cor}
%}
 
\vskip3pt
In \cite[Theorem 30]{Rueda},   Rueda et al. proved that a linear differential polynomial parametrization $(P(u),Q(u))\in\ff\{u\}^2$ is proper if and only if the  differential resultant $\delta\text{-}\Res(x-P(u),y-Q(u)) \neq 0$. 
However, this result is not valid for  linear differential rational parametrizations. 
We give a non-example as follows.

\begin{example}
 Let $\mathcal{P}(u)=(\frac{u''+1}{u},\frac{u''+1}{u})$. 
 Obviously, $\mathcal{P}(u)$ is not proper.
 But the differential resultant $\delta\text{-}\Res\big(xu-u''-1,yu-u''-1\big)=(y-x)^{3} \neq 0.$
 Note that $\ord((y-x)^{3} )<\ord(\mathcal{P}(u))$.
\end{example}

Below, we shall give a characterization of properness for linear differential rational parametrizations with the use of differential resultant
combined with the order property.
Before that, we need some preparations by studying a particular differential remainder sequence.

Given a linear differential rational parametrization $\mathcal P(u)=\big(\frac{P_1(u)}{Q_1(u)}, \frac{P_2(u)}{Q_2(u)}\big)$ with $\gcd(P_i,Q_i)=1$ and $m_i=\ord(\frac{P_i(u)}{Q_i(u)})\geq0$ for $i=1,2$. 
Without loss of generality, suppose $m_1\geq m_2$.
Denote
\begin{equation}
f_1(x, y, u)= P_1(u)-xQ_1(u), f_2(x, y, u)= P_2(u)-yQ_2(u).\end{equation}
Fix the elimination ranking $\mathscr R: x<y<u$. 
Let $f_3(x, y, u)=\delta$-$\prem(f_1,f_2)$ be the Ritt-Kolchin remainder of $f_1$ with respect to $f_2$ under $\mathscr R$.  
Since $f_1\notin\sat(f_2)$,  $f_3\neq 0$, $\ord_uf_3<m_2$ and $\ord_yf_3\leq m_1-m_2$.
And $Q_1\notin\sat(f_2)$ implies $\ord_xf_3=0$.
% \textcolor{blue}{(If $\ord_xf_3 <0$, then the fact that $f_3 \in [f_1,f_2]$ yields ${Q_1}^k f_3 \in [f_2] \subset \sat(f_2)$ for some $k \in \N$. Thus $f_3 \in \sat(f_2)$, contradicting $\ord_uf_3 < \ord_uf_2$)}, 
If $\ord_uf_3\geq0$, let $f_4(x, y, u)=\delta$-$\prem(f_2,f_3)$. 
Then, we have 
\begin{equation}   \label{eq-f4}
\ord_uf_4<\ord_uf_3,\,\ord_xf_4 \leq m_2- \ord_uf_3.
\end{equation} 
 If $\ord_uf_4 \geq 0$, then let  $f_5=\delta$-$\prem(f_3,f_4)$.
Continue the differential reduction process when  $\ord_uf_{l-1}\geq0$ until we get $f_l=\delta$-$\prem(f_{l-2},f_{l-1})\in \ff\{x,y\}$ for some $l\in\mathbb N$.
Then, we obtain a sequence of differential remainders 
\begin{equation} \label{eq-remaindersequence}
f_1(x,y,u), f_2(x,y,u), f_3(x,y,u),\cdots,f_{l-1}(x,y,u),f_{l}(x,y).
\end{equation}

\begin{lemma} \label{lm-reductionsequence}
The obtained sequence (\ref{eq-remaindersequence})
 satisfies the following properties:
 \begin{itemize}
 \item[1)] For $2 \leq i \leq l-1$, $\ord_uf_i+\ord_xf_{i+1} \leq m_2$ and $\ord_uf_i+\ord_yf_{i+1} \leq m_1$.
 \item[2)] For each $ i \geq 2$,   $f_i$ has  the following representation form 
\begin{equation} \label{eq-remainderform}
f_{i+1}=\sum_{k=0}^{m_2-\ord_uf_{i}}A_{i,k}(x,y)f_1^{(k)}+\sum_{j=0}^{m_1-\ord_uf_{i}}B_{i,j}(x,y)f_2^{(j)}
\end{equation} 
 where  {$A_{i,k}, B_{i,j}\in \ff\{x,y\}$}, and  {$B_{i,m_1-\ord_uf_{i}}\neq0, A_{i,m_2-\ord_uf_i}\neq0$ are  products of separants of $f_k\,(k \leq i)$}.   
 \end{itemize}
  \end{lemma}
\proof We shall show 1) and 2) by  induction on $i$. 
First note that  $\ord_uf_2+\ord_xf_{3}=m_2$ and $\ord_uf_2+\ord_yf_{3} \leq m_1$.
By the differential reduction formula for $f_3=\delta$-$\prem(f_1,f_2)$,
 there exist $a\in\mathbb N$ and $C_k  {\in \ff \{x,y\}}$ such that  $f_3=(\S_{f_2})^af_1-(\S_{f_2})^{a-1}\S_{f_1}f_2^{(m_1-m_2)}-\sum_{k=0}^{m_1-m_2-1}C_kf_2^{(k)}$.  So  both 1) and 2) holds for $i=2$. 
 
Now suppose  1) and 2) holds for $i\leq j\, ({\,\geq 2})$. We consider the case for $i=j+1$.
Since both $f_1$ and $f_2$ are linear differential polynomials in $u$, 
all the $f_i\,(i\leq l-1)$ are linear in $u$ and its derivatives and thus $\ord_uf_i<\ord_uf_{i-1}$.
By the induction hypothesis, $\ord_uf_{j}+\ord_xf_{j+1} \leq m_2$, and $f_{j+1}$ has a representation form as (\ref{eq-remainderform}) with $B_{j,m_1-\ord_uf_{j}}\neq0 {\,,A_{j,m_2-\ord_uf_{j}}\neq0}$.
Since $f_{j+2}=\delta$-$\prem(f_{j},f_{j+1})$, $f_{j+2}$ is a linear combination of $f_{j}$ and $f_{j+1}, f_{j+1}',\ldots, f_{j+1}^{(s)}$ with  coefficients in $\ff\{x,y\}$ and in particular the nonzero coefficient for $f_{j+1}^{(s)}$ is a product of separants of $f_j, f_{j+1}$ where $s=\ord_uf_j-\ord_uf_{j+1}$.
Thus, $f_{j+2}$ has a representation form as (\ref{eq-remainderform}) with $B_{j+1,m_1-\ord_uf_{j+1}}\neq0 {\,,A_{j+1,m_2-\ord_uf_{j+1}}\neq0}$ being products of the separants of $f_k\,(k\leq j+1)$.
And  \begin{eqnarray}
 \ord_xf_{j+2} &\leq& \max\{\ord_xf_{j+1}+\ord_uf_{j}-\ord_uf_{j+1}, \ord_xf_{j}\} \nonumber \\
 &\leq& \max\{m_2-\ord_uf_j+\ord_uf_{j}-\ord_uf_{j+1}, m_2-\ord_uf_{j-1}\}  \nonumber \\
  &\leq& m_2-\ord_uf_{j+1}. \nonumber
 \end{eqnarray}
So $\ord_uf_{j+1}+\ord_xf_{j+2} \leq m_2$. 
Similarly, $\ord_uf_{j+1}+\ord_yf_{j+2} \leq m_1$ can be shown.  
Thus, 1) and 2) are proved by induction.
\qed

\vskip5pt
Now, we are ready to propose the main theorem.
\begin{theorem} \label{th-propercriteria}
Let $\mathcal P(u) =(\frac{P_1(u)}{Q_1(u)},\frac{P_2(u)}{Q_2(u)})$ be a linear differential rational parametrization with $m_i=\ord(\frac{P_i(u)}{Q_i(u)})\geq 0$. 
Then $\mathcal{P}(u)$ is proper if and only if
$$R:=\delta\text{-}\Res\big(xQ_1(u)-P_1(u),yQ_2(u)-P_2(u)\big) \neq 0  \text{\,\, and\,\, }  {\ord_xR=\ord\big(\frac{P_2}{Q_2}\big), \ord_yR=\ord\big(\frac{P_1}{Q_1}\big)}.$$
\end{theorem}

\begin{proof}
``$\Rightarrow$": It follows from Theorem \ref{d-resultant}.\\
``$\Leftarrow$":  Without loss of generality, suppose $m_1 \geq m_2$.
If $m_2=0$, then $\ff\langle\frac{P_2(u)}{Q_2(u)}\rangle=\ff\langle u\rangle$ and thus  $\mathcal P(u)$ is proper.
So we only need to consider the case when $m_2\geq 1$. 

 Let $f_1=xQ_1(u)-P_1(u),  f_2=yQ_2(u)-P_2(u)\in \ff\{u,x,y\}$. 
  Fix the elimination ranking $\mathscr R: x<y<u$.
 Do the differential reduction process as in  Lemma \ref{lm-reductionsequence} under $\mathscr R$ and we consider the obtained sequence 
 $$f_1(x,y,u), f_2(x,y,u), f_3(x,y,u),\cdots,f_{l-1}(x,y,u),f_{l}(x,y).$$
By Lemma \ref{lm-reductionsequence}, there exist $a,a_i,b_j \in \ff\{x,y\}$ with $a \neq 0{\, , b_{m_2-\ord_uf_{l-1}} \neq 0}$ such that 
 \begin{equation}\label{eq-fl}
 f_l(x,y)=a(x,y)f_2^{(m_1-\ord_uf_{l-1})}+\sum_{i=0}^{m_1-\ord_uf_{l-1}-1}a_i(x,y)f_2^{(i)}+\sum_{j=0}^{m_2-\ord_uf_{l-1}}b_j(x,y)f_1^{(j)}.
 \end{equation}

 ({\bf Claim A}) $f_l$ and $f_{l-1}$ satisfy the following properties:
 \begin{itemize} 
 \item[1)] $f_l(x,y) \neq 0$,  $\ord_xf_l=m_2$  {and $\ord_yf_l=m_1$}.
 \item[2)] $f_{l-1}(x,y,u)=g(x,y)u+h(x,y)$,where $g,h \in \ff\{x,y\} \backslash \{0\}$.
\end{itemize}

We now proceed to prove {\bf Claim A}. 
We first show that $f_l\neq0$ and $\ord_uf_{l-1}=0$. 
Suppose the contrary that $f_l=0$.
Then by (\ref{eq-fl}), $f_1,f_1',\cdots,$ $f_1^{(m_2-\ord_uf_{l-1})},f_2,f_2',\cdots,f_2^{(m_1-\ord_uf_{l-1})}$ are linearly dependent over $\ff\langle x,y \rangle$.
As a consequence, $f_1,f_1',\cdots,f_1^{(m_2)},f_2,f_2',\cdots,f_2^{(m_1)}$ are linearly dependent over $\ff\langle x,y \rangle$,  which contradicts the fact that $\delta\text{-}\Res(f_1,f_2) \neq 0$. Thus $f_l(x,y) \neq 0$. 
 And by (\ref{eq-fl}), $1$ can be written as a linear combination of $f_1,f_1',\cdots,$ $f_1^{(m_2-\ord_uf_{l-1})},f_2,f_2',\cdots,f_2^{(m_1-\ord_uf_{l-1})}$ over $\ff\langle x,y \rangle$ with a nonzero coefficient for $f_2^{(m_1-\ord_uf_{l-1})}$. 
If $\ord_uf_{l-1}>0$,   then by differentiating on the both sides of this identity form, 
we obtain that $f_1,f_1',\cdots,f_1^{(m_2)}$, $f_2,f_2',\cdots,f_2^{(m_1)}$ are linearly dependent over $\ff\langle x,y \rangle$, which leads to a contradiction. 
So $\ord_uf_{l-1}=0$ and $f_{l-1}=g(x,y)u+h(x,y)$ for some $g,h \in \ff\{x,y\}$ with $g\neq0$.

 It remains to show  that $\ord_xf_l=m_2$, $\ord_yf_l=m_1$ and $h\neq0$. 
 Denote $m=m_1+m_2+2$.
Let $M\in \ff\langle x,y\rangle^{m\times m}$ be the resultant matrix of $f_1^{(m_2)},\ldots,f_1',f_1,f_2^{(m_1)},$ $\ldots,f_2',f_2$ w.r.t. the variables $u^{(m_1+m_2)}, u^{(m_1+m_2-1)},$ $\ldots,u',u$.
We perform row operations on $M$ using the $a_i$ and $b_j$ in (\ref{eq-fl}) as follows.
For $i\neq m_2+2$, add a multiple $c_i$ of the $i$-row to the $(m_2+2)$-th row of $M$ successively and denote the obtained matrix by $M_1$, where for $i=1,\ldots,m_2+1$, $c_i=b_{m_2+1-i}/a$; and for $i\geq m_2+3$, $c_i=a_{m_1+m_2+2-i}/a$.
By (\ref{eq-fl}), the $(m_2+2)$-th row of $M_1$ becomes $(0\,\,0\,\,\cdots\,\,0 \,\,f_l/a)$.
Thus, \begin{equation} \label{eq-res-fl}
\delta\text{-}\Res(f_1,f_2)=\det(M)=\det(M_1)={(-1)^{m_1}}c\cdot f_l/a\cdot \det(M_2),
\end{equation} where $c\neq0$ is the coefficient of $f_1^{(m_2)}$ in $u^{(m_1+m_2)}$,
and $M_2$ is the matrix obtained by deleting the $1$-th row, the $(m_2+2)$-row, the $1$-th column and the last column. 
Obviously,  $\ord_x\det(M_2)<m_2$ and $\ord_xc\leq0$. 
By Lemma \ref{lm-reductionsequence}, $\ord_xa(x,y)<m_2$. 
Since $\ord_x\delta\text{-}\Res(f_1,f_2)=\ord(\frac{P_2}{Q_2})=m_2$,
 by (\ref{eq-res-fl}), we have $\ord_xf_l=m_2$. 
Similarly, $\ord_yf_l=m_1$ follows from the facts that $\ord_y\det(M_2)<m_1$, $\ord_yc=-\infty$, $\ord_ya(x,y)<m_1$
and $\ord_y\delta\text{-}\Res(f_1,f_2)=m_1$.
We finish the proof of {\bf Claim A} by showing $h\neq0$.
If $h=0$, then $f_l=g^kf_{l-2}(x,y,0)$ for some $k\in\mathbb N$ and thus $\ord_xf_l\leq \max\{\ord_xg, \ord_xf_{l-2}\}<m_2$ by Lemma \ref{lm-reductionsequence}, a contradiction.
Thus $h\neq0$  and  {\bf Claim A}  is proved.

\vskip2pt
Since $\det(M)$ is linear in $x^{(m_2)}$ and $\ord_xf_l=m_2$, by (\ref{eq-res-fl}),  $f_l$ is  linear in $x^{(m_2)}$.
Recall that $\ord_xf_{l-1}<m_2$.
Thus, $ \mathcal A=f_l, f_{l-1}$ constitutes an irreducible autoreduced set w.r.t. the elimination ranking $\mathscr R_1: y<x<u$,
and $\mathcal P=\sat(\mathcal A)$ is a prime differential ideal with a characteristic set $ \mathcal A$ under $\mathscr R_1$.
We shall show $[f_1,f_2]:(Q_1Q_2)^\infty \subset\mathcal P$ by proving i) $f_1,f_2 \in \mathcal P$ and ii) $Q_1,Q_2 \notin \mathcal P$.

To show i), by Lemma \ref{lm-reductionsequence}, for each $1 \leq i \leq l-1$, $\ord_xf_i <m_2$ and thus the separant $\S_{f_i} \notin \mathcal P$. By the reduction process, for each $1 \leq i \leq l-2$, 
there exists $k_i \in \N$ such that  $\S_{f_{i+1}}^{k_i}f_i \equiv f_{i+2} \,\text{mod}\,[f_{i+1}]$. 
Therefore, $f_{l-2} \in \mathcal P$ and consequently $ f_2, f_1 \in \mathcal P$.  
 
To show ii),  let $M_3 \in \ff\{x,y\}^{m\times m}$ be the resultant matrix of $Q_1,f_1^{(m_2-1)},\ldots,f_1'$, $f_1,f_2^{(m_1)},\ldots,$ $f_2',f_2$ w.r.t.  $u^{(m_1+m_2)},\ldots, u',u$. 
Then $\det(M_3)=$ coeff$(R,x^{(m_2)}) \neq 0$,
  for   $R$ is linear in $x^{(m_2)}$. 
  Therefore, $Q_1,f_1^{(m_2-1)},\ldots,f_1'$, $f_1,f_2^{(m_1)},\ldots, f_2',f_2$ are linearly independent over $\ff(x^{[m_2-1]},y^{[m_1]})$, and $$\text{Span}_{\ff(x^{[m_2-1]},y^{[m_1]})}(Q_1,f_1^{(m_2-1)},\ldots,f_1', f_1,f_2^{(m_1)},\ldots, f_2',f_2)=\text{Span}_{\ff(x^{[m_2-1]},y^{[m_1]})}(1,u,\cdots,u^{(m_1+m_2)}).$$
{The representation of 1 in terms of $Q_1,f_1^{(m_2-1)},\ldots,f_2$ yields a nonzero differential polynomial ${G} \in  [Q_1,f_1,f_2]\cap\mathcal \ff\{x,y\}$ with $\ord_x{G} <m_2$. 
If $Q_1\in \mathcal P$, then ${G} \in \mathcal P\cap \mathcal F\{x,y\}=\sat(f_l)$, which is impossible. 
Thus, $Q_1 \notin \mathcal P$. 
%Since we also have $\ord_yf_{l-1}<m_1$, $\ord_yf_l=m_1$ and $f_l$ is  linear in $y^{(m_1)}$,
%$f_l, f_{l-1}$ constitute an irreducible  autoreduced set w.r.t. the elimination ranking $\mathscr R: x<y<u$.
%Since  $f_l$ is a characteristic set of $\mathcal P\cap\ff\{x,y\}=\sat(f_l)$ w.r.t. any ranking and $f_{l-1}=g(x,y)u+h(x,y)$,
%$f_l, f_{l-1}$ is also a characteristic set of $\mathcal P$ w.r.t. $\mathscr R$.
Note that $\frac{\partial^2 R}{\partial  x^{(m_2)}\partial  y^{(m_1)}}=\det(M_4)=0$, where $M_4$ is the resultant matrix of $Q_1,f_1^{(m_2-1)},\ldots,f_1'$, $f_1,Q_2,f_2^{(m_1-1)}\ldots,$ $f_2',f_2$ w.r.t.  $u^{(m_1+m_2)},\ldots, u',u$. 
So by (\ref{eq-res-fl}),  if $B$ is the irreducible factor of $f_l$ effectively involving $x^{(m_2)}$, then $\ord_yB=m_1$ and $B$ is  linear in $y^{(m_1)}$. 
Thus, $\mathcal P\cap\ff\{x,y\}=\sat(f_l)=\sat(B)$ and $B$ is a characteristic set of $\sat(f_l)$ w.r.t. any ranking.
Since coeff$(R,y^{(m_1)}) \neq 0$,
repeating the above procedures, we obtain some $H\in[Q_2,f_1,f_2]\cap\mathcal \ff\{x,y\}$ with $\ord_y{H} <m_1$,
 and consequently $Q_2 \notin \mathcal P$.
Thus, $\big([f_1,f_2]\colon(Q_1Q_2)^\infty\big) \subset \mathcal P$.
}

%Similarly, \textcolor{blue}{since $\det(M)$ is linear in $y^{(m_1)}$, by (\ref{eq-res-fl}),  $\ord_yf_l=m_1$ and $f_l$ is  linear in $y^{(m_1)}$. Recall that $\ord_yf_{l-1}<m_1$. Thus, $ \mathcal B=f_l, f_{l-1}$ constitutes an irreducible differential ascending chain w.r.t. the elimination ranking $\mathscr R_1$ with $x<y<u$,
%and $\mathcal P_1=\sat(\mathcal B)$ is a prime differential ideal with $ \mathcal B$ as its characteristic set under $\mathscr R_1$. From the differential reduction process, $\text{coeff}(f_l, x^{(m_2)})$ has order in $y$ lower than $m_1$ and $\text{coeff}(f_l, y^{(m_1)})$ has order in $x$ lower than $m_2$. Therefore, $\frac{\partial f_l}{\partial x^{(m_2)}} \notin \mathcal P_1$ and $\frac{\partial f_l}{\partial y^{(m_1)}} \notin \mathcal P$, which implies $\mathcal P= \mathcal P_1$. Since coeff$(\delta\text{-}\Res(f_1,f_2),y^{(m_1)}) \neq 0$, there exists $N \in \ff\{x,y\}\backslash\{0\}$ with $\ord_yN<m_1$ such that $N \in [Q_2,f_1,f_2]$. If $Q_2 \in \mathcal P$, then $N \in [Q_2,f_1,f_2]\,\cap\,\ff\{x,y\} \subset \mathcal P=\mathcal P_1$, which is impossible. Thus $Q_2 \notin \mathcal P$.} 
%$Q_2 \notin \mathcal P$. Thus, $[f_1,f_2]:(Q_1Q_2)^\infty  \subset \mathcal P$. 

Suppose $\mathcal P(u)$ is a differential rational parametrization of  $A(x,y) \in \ff\{x,y\}$.
Then $A \in \big([f_1,f_2]\colon(Q_1Q_2)^\infty\big) {\,\cap\, \ff \{x,y\}}\subset \mathcal P\cap\mathcal F\{x,y\}=\sat(f_l)$. 
Since $\ord_xA \leq m_2$ by Proposition \ref{prop-orderdifference}, $\ord_xA= m_2$ and {$\ord_yA= m_1-m_2+\ord_xA=m_1$ follows.} 
By Theorem \ref{proper-rational}, $\mathcal P(u)$ is proper.
\end{proof}

%\textcolor{red}{As a direct consequence of the proof of Theorem \ref{th-propercriteria}, another criterion to test properness for  linear differential rational parametrizations  could be stated as follows. 
%\begin{cor}
%Let $\mathcal P(u) =(\frac{P_1(u)}{Q_1(u)},\frac{P_2(u)}{Q_2(u)})$ be a linear differential rational parametrization with $m_i=\ord(\frac{P_i(u)}{Q_i(u)})\geq 0$. 
%Then  $\mathcal{P}(u)$ is proper if and only if both {the resultant matrix of $Q_1,f_1^{(m_2-1)},\ldots,f_1', f_1,$ $f_2^{(m_1)},\ldots, f_2',f_2$ w.r.t.  $u^{(m_1+m_2)},\ldots, u',u$ } and the resultant matrix of $f_1^{(m_2)},\ldots,f_1', f_1,$ $Q_2, f_2^{(m_1-1)},\ldots, f_2',f_2$ w.r.t.  $u^{(m_1+m_2)},\ldots, u',u$ are of full rank. \end{cor} 
%}

Applying Theorem \ref{th-propercriteria}, we can devise an algorithm to decide whether a given linear differential rational parametrization is proper or not and in the affirmative case, to compute the implicit equation. 
Precisely, given the following linear differential rational parametric equations:
\[ \left\{\begin{array}{l}
x=\frac{P_1(u)}{Q_1(u)} \\
\quad \\
y=\frac{P_2(u)}{Q_2(u)}
\end{array},
\right. 
 \]
 we first compute the differential resultant $R=\delta$-$\Res(Q_1(u)x-P_1(u),Q_2(u)y-P_2(u))$ w.r.t. $u$.
 In the cases when (1) $R=0$, or (2) $R\neq0$ and $(\ord_xR,\ord_yR)\neq(\ord(\frac{P_2(u)}{Q_2(u)}),\ord(\frac{P_1(u)}{Q_1(u)})$, this parametrization is not proper.
 Otherwise,  $R\neq0$, $\ord_xR=\ord(\frac{P_2(u)}{Q_2(u)})$ and $\ord_yR=\ord(\frac{P_1(u)}{Q_1(u)})$, then this parametrization is  proper
 and the implicit equation is the main factor of $R$.

 The correctness of this algorithm is guaranteed by Theorem  \ref{th-propercriteria}.
To compute the differential resultant $R$, we just need to compute the determinant of a matrix of size $m_1+m_2+2$ with $m_i=\ord(\frac{P_i}{Q_i})$. So the complexity of this algorithm is $(m_1+m_2+2)^\omega$, which is much more efficient than the implicitization algorithms using characteristic sets.  Here, $\omega$ is the exponent of matrix multiplication.
However, we should point out that when the given paramatrization is not proper, the current algorithm fails to compute the implicit equation while the characteristic set method can always produce the implicit equation.

\begin{remark}
In this section, we are interested in proper linear differential rational parametrizations.
However, even if a uniration differential curve has a linear differential rational parametrization,
it may happen that it does not have a proper linear differential rational parametrization.
For example, let  $A=(x'x+x^{3}-2x^{2}-x)y'+(xy+y^{2})x''+3y^{2}x'x+y^{2}x^{3}-2yx'^{2}+6yx'x+3yx'+4yx^{3}+3yx^{2}-2xy-y+5x^{3}+2x^{2}$. Then $\mathcal P(u)=(\frac{u''}{u'+u'''},\frac{u'+2u''}{u^{(4)}})$ is a linear  differential rational parametrization of $(\C, A)$ and $\mathcal Q(u)=\big(\frac{u}{u'+u^2+1},\frac{2u+1}{u''+3uu'+u^{3}}\big)$ is a proper  differential rational parametrization of $(\C,A)$. 
By Theorem \ref{th-mobius},  any proper  differential rational  parametrization of $(\C,A)$ is of degree greater than $1$.
\end{remark}

 \vskip2pt
 We conclude this section by giving a property of the inversion map of  proper linear differential rational parametrization.
\begin{cor}
Let  $(\C,A)$ be a unirational differential curve with a proper linear differential rational parametrization $\mathcal{P}(u)=(\frac{P_1(u)}{Q_1(u)},\frac{P_2(u)}{Q_2(u)})$.
Let $m_i=\ord(\frac{P_i}{Q_i}) \geq 1$ for $i=1,2$. 
Suppose $\mathcal U: \C  -\hskip-.1truecm -\hskip-.1truecm\to\mathbb A^1$ given by $U(x,y)\in\ff\langle x,y\rangle$ is  the inversion of  $\mathcal P(u)$. 
Then $\ord_xU<m_2 , \ord_yU<m_1$.
 \end{cor}
\begin{proof}
By the proof of Theorem \ref{th-propercriteria}, $f_{l-1}$ is linear in $u$ with $\ord_xf_{l-1}<m_2$ and $\ord_yf_{l-1}<m_1$.  
Thus, $u \in \ff\big(\frac{P_1}{Q_1}, (\frac{P_1}{Q_1})', \cdots, (\frac{P_1}{Q_1})^{(m_2-1)}, \frac{P_2}{Q_2}, (\frac{P_2}{Q_2})',\cdots, (\frac{P_2}{Q_2})^{(m_1-1)}\big).$
 \end{proof}

% \begin{algorithm}   \label{propercriteria}
%  \caption{\bf ---Implicit($\mathcal P(u)$)} \smallskip
%  \Inp{A linear differential rational parametrization $\mathcal P(u)=\big(\frac{P_1(u)}{Q_1(u)},\frac{P_2(u)}{Q_2(u)}\big)$.}\\
%  \Outp{***.}\medskip
%
%  \noindent
%  1.  
%  
%    3. While $R=0$ do\\
% 
%  4. Return $R$.\medskip
%\end{algorithm}

\section{Rational parametrization for linear differential curves}
We now deal with the paramerization problem for differential curves, which asks for criteria to decide algorithmically whether an implicitly given differential curve is unirational or not, and in the affirmative case, to return a differential rational parametrization.
In general, it is a very difficult problem.
In this section, we start from the simplest nontrivial case by considering the  unirationality problem of linear differential curves.

\begin{definition}
Let $(\C,A(x,y))$ be an irreducible differential curve. We call $\C$ a linear differential curve if $A$ is a linear differential polynomial in $\ff\{x,y\}$.
\end{definition} 

 In the algebraic case, by Gaussian elimination, we know each linear variety is unirational and has a polynomial parametric representation. However, in the differential case, even a linear differential curve might not be unirational as shown in Example \ref{ex-1} (2).

Every linear differential polynomial $A(x,y)\in \ff\{x,y\}$ is of the form $$A=L_1(x)+L_2(y)+a$$
for some linear differential operators $L_1,  L_2\in \ff[\delta]$ where at least one $L_i$ is nonzero and $a\in \ff$.
We shall show in Theorem \ref{th-linear-unirational} that a necessary and sufficient condition for $A(x,y)$ to be unirational is that
the greatest common left divisor of $L_1, L_2$ is 1. 
Before that, we recall the extended left Euclidean algorithm given by Bronstein et al. in \cite[pp.14-15]{Bronstein}.

\vskip2pt
In general, $\ff[\delta]$ is a non-commutative domain and there exist the left and the right Euclidean divisions  in $\ff[\delta]$. 
Given $L_1, L_2\in \ff[\delta]$ with $L_2\neq0$,  by the left Euclidean division, we obtain  $Q, R \in \ff[\delta]$ with $\deg(R)<\deg(L_2)$ satisfying $L_1=L_2Q+R$, where $Q$ and $R$  are called respectively the {\it left-quotient} and the {\it left-remainder}  of $L_1$ w.r.t. $L_2$, denoted by  $\text{lquo}(L_1,L_2)$ and $\text{lrem}(L_1,L_2)$.  
 If $R=0$, then $L_2$ is called a left divisor of $L_1$, 
  and correspondingly, $L_1$ is called a right multiple of $L_2$.
  % and correspondingly $Q$ is called a right divisor  of $L_1$.
 A common left divisor  of $L_1$ and $L_2$ with the highest degree is called  a {\it greatest common left divisor} of $L_1$ and $L_2$. There exists a unique, monic (i.e., reduced in Ore's sense), greatest common left divisor, denoted by  $\text{gcld}(L_1,L_2)$.   A common right multiple of $L_1, L_2$ of minimal degree is called a {\it least common right multiple}.
And we have analogous notions for right Euclidean divisions.
Below, we restate the extended left Euclidean algorithm {\bf ELE}$(L_1,L_2)$ for later use and list its basic properties in Proposition \ref{left}, which  were given in {\rm \cite[pp.14-15]{Bronstein}.} 

\vskip5pt
{\bf Left Euclidean Algorithm:}  {\bf ELE}$(L_1,L_2)$

\hskip.1truecm \textbf{Input:} {$L_1, L_2 \in \ff[\delta]$.} \vskip1pt
\hskip.1truecm \textbf{Output:} {The tuple $(R_{n-1}, A_n, B_n, A_{n-1}, B_{n-1})\in  \ff[\delta]^5$.} \vskip1pt
   \hspace*{.5em}   1.  $R_0:=L_1, A_0:=1, B_0:=0$; \vskip2pt
   \hspace*{1.5em}    $R_1:=L_2, A_1:=0, B_1:=1$;\vskip2pt
  \hspace*{1.5em}    $i:=1.$ \vskip2pt
   \hspace*{.5em}  2. While $R_i \neq 0$ do \vskip2pt
 \hspace*{1.5em} $i:= i+1$; \vskip2pt
 \hspace*{1.5em} $Q_{i-1}:= \text{lquo}(R_{i-2},R_{i-1})$;\vskip2pt
 \hspace*{1.5em} $R_{i}:= \text{lrem}(R_{i-2},R_{i-1})$;\vskip2pt
 \hspace*{1.5em} $A_i :=A_{i-2}-A_{i-1}Q_{i-1}$; \vskip2pt
 \hspace*{1.5em} $B_i:= B_{i-2}-B_{i-1}Q_{i-1}$. \vskip2pt
 \hspace*{.5em} 3. {$n:=i$}, and {\bf return} $(R_{n-1}, A_n, B_n, A_{n-1}, B_{n-1})$.

\vskip5pt
\begin{prop}\label{left}
Let $L_1, L_2 \in \ff[\delta]$. The algorithm {\bf ELE}$(L_1,L_2)$ can be used to compute a greatest common left divisor and a least common right multiple of $L_1,L_2$ and the obtained sequences $A_i,B_i,R_i$ satisfy the following properties:
\begin{enumerate}
\item[$(1)$.] $R_{n-1}$ is a greatest common left divisor of $L_1,L_2$; 
\item[$(2)$.]$L_1A_n=-L_2B_n$ is a least common right multiple of $L_1,L_2$;
\item[$(3)$.]$R_i=L_1A_i+L_2B_i\,$ for $0 \leq i \leq n$; 
\item[$(4)$.]$\deg(A_i)=\deg(L_2)-\deg(R_{i-1})$, $\deg(B_i)=\deg(L_1)-\deg(R_{i-1})$\, for $2 \leq i \leq n$.
\end{enumerate}
\end{prop}

The following result given in \cite{Rueda}  will be used to derive Theorem \ref{th-linear-unirational}.
 \begin{lemma}\label{rueda}  {\rm \cite[Theorem 30]{Rueda}}
Let $L_1, L_2\in\ff[\delta]$. Then $\ff\langle L_1(u), L_2(u) \rangle =\ff \langle u \rangle$ if and only if $\text{\rm gcrd}(L_1, L_2)=1$ {{(i.e., a greatest common right divisor of $L_1, L_2$ belongs to $\ff\backslash\{0\})$}}.
\end{lemma} 
 \vskip4pt
 \begin{theorem}\label{th-linear-unirational}
 Let $F=L_1(x)+L_2(y)+a\in\ff\{x,y\}\backslash\ff$
 with $L_1, L_2\in\ff[\delta]$.
 Then, \begin{center}
 $(\C,F)$ is unirational if and only if $\text{\rm{gcld}}(L_1, L_2)=1$.
 \end{center}
 Furthermore, each unirational linear differential curve has a proper linear differential polynomial parametrization. 
 \end{theorem}
 
 \begin{proof}
For the necessity,  suppose $(\C,F)$ is unirational.  
If $\text{\rm{gcld}}(L_1, L_2) \neq1$, then there exist $L \in \ff[\delta]\backslash \ff$ and $ L_3, L_4 \in \ff[\delta]$ such that $L_1=LL_3, L_2=LL_4$. Then $F=L(L_3(x)+L_4(y))+a$ with $\ord(L_3(x)+L_4(y)) < \ord(F)$. Since there exists $\mathcal P(u) \in \ff \langle u \rangle ^2$ such that $\sat(F)=\mathbb I(\mathcal P(u))$, we have $F(\mathcal P(u))=0$, which implies that $L_3(\mathcal P(u))+L_4(\mathcal P(u)) \in \ff$. 
Thus, $A:=L_3(x)+L_4(y) -b \in \sat(F)$ for some $b \in \ff$. 
Since $\ord(A) < \ord(F)$,  this leads to a contradiction.  So $\text{\rm{gcld}}(L_1, L_2) =1$.
 
     \vskip2pt
To show the sufficiency,  suppose $\text{\rm{gcld}}(L_1, L_2) = 1$.  
By performing the algorithm {\bf ELE($L_1,L_2$)}, we obtain $A_i,\,B_i, \,R_i  \in \ff[\delta]$ satisfying the properties given in Proposition \ref{left}. 
In particular, $L_1A_n=-L_2B_n$ is a least common right multiple of $L_1$ and $L_2$.
 The fact that $\text{\rm{gcld}}(L_1, L_2) = 1$ yields $c:=R_{n-1} \in \ff \backslash\{0\}$, and consequently 
 $\deg(A_n)=\deg(L_2), \deg(B_n)=\deg(L_1)$. 
 Let $$\mathcal P(u)=\Big(A_n(u)+A_{n-1}(-a/{c}),B_n(u)+B_{n-1}({-a}/{c})\Big)\in \ff\{u\}^2.$$
 We shall show that  $\mathcal P(u)$ is a proper linear differential polynomial parametrization of $(\C,F)$. 
      
      \vskip2pt
  We first prove $\ff\langle\mathcal P(u)\rangle=\ff\langle u\rangle$. 
  Since $\ff\langle\mathcal P(u)\rangle=\ff\langle A_n(u), B_n(u)\rangle$,
  by Lemma \ref{rueda}, it suffices to prove that gcrd$(A_n, B_n)=1$. 
  If $\text{\rm{gcrd}}(A_n, B_n) \neq 1$, there exists $C(\delta) \in \ff[\delta]\backslash\ff$ such that $A_n=C_1(\delta)C(\delta), B_n=C_2(\delta)C(\delta)$ for some $C_1, C_2 \in \ff[\delta]$. Since $L_1A_n=-L_2B_n$, we obtain $L_1C_1=-L_2C_2$ is also a common right multiple of $L_1,L_2$, which contradicts the fact that  $L_1A_n$ is a least common right multiple of $A,B$. 
  Therefore,  $\ff\langle\mathcal P(u)\rangle=\ff\langle u\rangle$.

     \vskip2pt
By Theorem \ref{th-order}, there exists an irreducible differential polynomial $G(x,y) \in \ff\{x,y\}$ with $\ord_xG=\deg(L_1), \ord_yG=\deg(L_2)$ such that $\mathbb I(\mathcal{P}(u))=\sat(G)$.  
Since $L_1A_{n-1}+L_2B_{n-1}=c$,  by acting this operator on $-a/c$, we have $L_1(A_{n-1}(\frac{-a}{c}))+L_2(B_{n-1}(\frac{-a}{c}))=-a$.
 So $F(\mathcal P(u))=L_1(A_n(u))+L_2(B_n(u))=(L_1A_n+L_2B_n)(u)=0$ and $F\in\sat(G)$ follows.
 Since $\ord(G)=\ord(F)$ and $F$ is linear, $F=eG$ for some $e\in \ff$ and $\mathbb I(\mathcal{P}(u))=[F]$.
 Thus, $\mathcal P(u)$ is a proper linear differential  polynomial parametrization of $F$ and $(\C,F)$ is unirational.
 The later assertion follows directly from the above proof.
 \end{proof}
 \vskip5pt
By the proof of Theorem \ref{th-linear-unirational}, we can devise an algorithm to determine whether an implicitly given  linear differential curve is unirational or not, and in the affirmative case, to construct a proper linear differential polynomial parametrization for it. 
\newpage
{\bf Algorithm Linear-Differential-Curve-Parametrization:}  {\bf LDCP}$(F)$
\vskip1pt
\hskip.1truecm \textbf{Input:} {$F=L_1(x)+L_2(y)+a\in\ff\{x,y\}\backslash\ff$ with $L_1, L_2\in\ff[\delta]$.} \vskip2pt
\hskip.1truecm \textbf{Output:} {A proper linear differential polynomial parametrization $\mathcal P(u)$ of $(\C, F)$, if it is \\\hspace*{5.5em} unirational; No, in the contrary case.} \vskip2pt
 \hspace*{.5em}  1.  Perform {\bf ELE}$(L_1,L_2)=(R_{n-1}, A_n, B_n, A_{n-1}, B_{n-1})$; \vskip2pt
 \hspace*{.5em}  2. If $R_{n-1} \notin \ff$, then {\bf return} No; \vskip2pt
 \hspace*{.5em} 3. {\bf Return} $\mathcal P(u)=\Big(A_n(u)+A_{n-1}(-a/{R_{n-1}}),B_n(u)+B_{n-1}({-a}/{R_{n-1}})\Big)$.
}
\vskip5pt
 Below, we give  examples to illustrate Theorem \ref{th-linear-unirational}.
\begin{example} \label{ex-linearunirational}
\begin{itemize}
\item[{\rm (1)}] Let $A=x''-y' \in \Q\{x,y\}$, then $L_1=\delta^2,L_2=\delta$ and $(\C,A)$ is not unirational.
\item[{\rm (2)}] Let $A=y'-x'-x \in \Q\{x,y\}$, then $L_1=-\delta-1,L_2=\delta$, $(\C,A)$ is unirational with a proper parametrization $(u',u+u')$.
\item[{\rm (3)}] Let $A=x'+x+ty'+(t+1)y \in \Q(t)\{x,y\}$ with $\delta =\frac{d}{dt}$.
  For this example,  $L_1=\delta+1, L_2=t\delta+t+1$. Since {$\gcld(L_1, L_2)=\delta+1$},
$(\C,A)$ isn't unirational.  
\item[{\rm (4)}] Let $A=tx'+tx+y'+y \in \Q(t)\{x,y\}$ with $\delta =\frac{d}{dt}$. Then $(\C,A)$ is unirational with a proper parametrization $(u'+u,-tu'+(1-t)u)$. 
Here  $L_1=t\delta+t, L_2=\delta+1$. Note that $\gcld(L_1, L_2)=1$ but $\gcrd(L_1, L_2)=\delta+1$.
\end{itemize}
\end{example}
% \begin{cor}
% Every unirational linear differential curve has a proper linear differential polynomial parametrization. 
% \end{cor}
%} 
%It is well-known that in the algebraic case, if two prime ideals with the same dimension have the inclusion relation, they are actually the same. However, this result doesn't hold true for the differential case. The natural question is that, if we add some additional conditions, such as both of the given two prime differential ideals define unirational differential curves, whether the inclusion relation could imply the equality. In the following, we give a positive answer to a special case.

\vskip3pt
{Given two differential polynomials $A, B\in\ff\{x,y\}$, we have shown  in Example \ref{ex-1} 2) that if $A$ is a proper derivative of $B$,
then $(\C,A)$ is not differentially unirational. 
If additionally both $A$ and $B$ are linear, we have a stronger result as follows.
 
\begin{prop}\label{linearinclude}
Let $A, B$ be two linear differential polynomials in $\ff\{x,y\}\backslash \ff$.  
\begin{itemize}
\item[1).] If $A\in [B]$ with $\ord(A)>\ord(B)$, then $(\C,A)$ is not differentially unirational.
\item[2).] Suppose $A$ is differentially unirational. If $[A]\subseteq[B]$, then $[A]=[B]$.
\end{itemize}
\end{prop}
\begin{proof}
Without loss of generality, suppose $\ord(A)=\ord_xA$ and set $s=\ord_xA-\ord_xB$.
Suppose  $A\in [B]$.
Since $A, B$ are linear,  there exists $a_i\in\ff$ with $a_{s}\neq0$ such that 
$$A=a_0\cdot B+a_1\cdot B'+\cdots+a_{s}\cdot B^{(s)}.$$
Clearly, the result 2) is a direct consequence of 1), so it suffices to show 1).
If $\ord(A)>\ord(B)$, then $s>0$ and $a_0+a_1\delta+\cdots+a_s\delta^s$ is a common left divisor of $L_1, L_2$, where $A=L_1(x)+L_2(y)+b\in\ff\{x,y\}$.
By Theorem \ref{th-linear-unirational}, $A$ is not  differentially unirational.
\end{proof}
%If $[A] \subsetneqq [B]$, then $\ord B < \ord A$. Since $A \in [B]$ and both $A, B$ are linear differential polynomials, by the differential reduction process, we may write $A$ as a linear combination of $B$ and its derivatives over $\ff$.  On the other hand, due to the fact that $(\C,A)$ is unirational, there exist $P(u), Q(u) \in \ff\langle u \rangle$ such that $\mathbb I(P(u), Q(u))=[A]$ . Thus, $B(P(u), Q(u)) \in \ff$, $B(x,y)-B(P(u), Q(u)) \in [A]$, which yields a contradiction!

\begin{remark} 
 Proposition \ref{linearinclude} shows that given two linear differential curves $\C_2\subseteq \C_1$, if $\C_1$ is unirational, then $\C_1=\C_2$.
However, in general, the inclusion of unirational differential curves doesn't imply the equality. 
For example, let $A=y'x-x'y+xy^{2}$ and $B=y$, then $(\C_2, B)\subsetneqq (\C_1, A)$.
Note that $(\C_1, A)$ and $(\C_2, B)$ are  unirational with prametrizations 
$(u',\frac{u'}{u})$ and $(u,0)$ respectively. 

\end{remark}

 {Now consider the implicitization of linear differential curves with given linear differential polynomial parametrization equations.
 In \cite[Sec. 8.1., Algorithm 2]{Rueda}, an algorithm was devised to compute the implicit equation of the linear differential curve using differential resultants.
We give an alternative method 
%to compute the implicit equation of the linear differential curve for a given linear differential polynomial parametrization 
based on the differential remainder sequence introduced in (\ref{eq-remaindersequence}). 
\begin{prop}
Let $\mathcal P(u)=(P_1(u),P_2(u))\in(\ff\{u\}\backslash\ff)^2$ be a linear differential polynomial parametrization.
Let $f_1=x-P_1(u),f_2=y-P_2(u)\in \ff\{u,x,y\}$. 
Suppose $$f_1(x,y,u),f_2(x,y,u),\cdots, f_{l-1}(x,y,u), f_l(x,y)$$ is the differential remainder sequence under the elimination ranking $\mathscr R: x<y<u$ obtained by the differential reduction process as in  Lemma \ref{lm-reductionsequence}. 
Then $\mathcal P(u)$ parametrizes $(\mathcal C,f_l(x,y))$
\end{prop}
\begin{proof}
%For the special cases that either $\mathcal P_1(u)=a_1 \in\ff$ or $\mathcal P_2(u)=a_2 \in \ff$, we have either 1) $A=x-a_1$ or 2) $A=y-a_2$. Now we consider the case that neither $P_1(u)$ nor $\mathcal P_2(u)$ belongs to $\ff$. 
Let $A(x,y)=0$ be the implicit equation of $x=P_1(u), y=P_2(u)$. 
That is, $A$ is linear with $\mathbb I(\mathcal P(u))=[A(x,y)]$. 
Since $f_{i-2}\equiv f_i \mod [f_{i-1}]$ for $i=3,\ldots,l$,
each $f_i\in[f_1, f_2]$ and $f_1,f_2\in[f_{l},f_{l-1}]$.
We first show that $f_l(x,y) \neq 0$. 
Suppose the contrary,  then $f_2, f_1 \in [f_{l-1}]$. 
Since $f_{l-1} \in [f_1,f_2]$,  we have $\ord_xf_{l-1} \geq 0$ or $\ord_yf_{l-1} \geq 0$,
which  contradicts the fact that $f_2,f_1\in [f_{l-1}]$.
So $f_l(x,y) \neq 0$. Without loss of generality,  suppose $\ord(P_1)\geq \ord(P_2)$.
 Two cases are considered:
\begin{itemize}
\item Case 1) $l=3$.  Here, $f_3(x,y) \in \mathbb I(\mathcal P(u))=[A]$. Since $\ord_xf_3 =0$ and $\ord_xA \geq 0$, $f_3=cA$ for some $c \in \ff$ and $\mathbb I(\mathcal P(u))=[f_3]$ follows.
 \item Case 2) $l \geq 4$.  
 Since $\ord_yf_i=\ord_uf_1-\ord_uf_{i-1}$ for $i\geq 3$, $\ord_yf_{i-1}<\ord_yf_{i}$ for $4 \leq i \leq l$. 
 Thus, $\mathcal A: f_l(x,y),f_{l-1}(x,y,u)$ is a characteristic set of the prime differential ideal $[f_l,f_{l-1}]$ under $\mathscr R$. Since $f_2, f_1 \in [f_l,f_{l-1}]$, $A(x,y) \in [f_1,f_2] \subseteq [f_{l},f_{l-1}]$. 
 Therefore, $\ord_yf_l \leq \ord_yA$. 
 Since $f_l \in \mathbb I(\mathcal P(u))=[A]$,  $f_l =cA$ for some $c \in \ff$ and $\mathbb I(\mathcal P(u))=[f_l]$.
 \end{itemize}
 Thus,  $\mathcal P(u)$ is a differential parametrization of the differential curve $(\C,f_l).$
\end{proof}
}

\begin{remark}
{  Example \ref{ex-linearunirational} (1) shows that the linear differential curve $\C=\mathbb V(x''-y')\subset\mathbb A^2$ is not unirational. 
However, if we allow differential rational parametrizations involving arbitrary constants as in \cite[Example 1.2.]{Gao2003}, 
then $\C$ has a parametrization of the form $x=u$ and $y=u'+c$ where $c$ is an arbitrary constant.
It is also an interesting topic  to study  generalized ``unirational" differential curves  with rational parametrizations involving arbitrary constants.} 

\end{remark}
% 
%The following result shows that a class of nonlinear differential polynomials of the special form can be differentially parametrized. 
%\begin{prop}
%Let $A=(p(x)+ay)^{(k)}+bx-c \in \ff\{x,y\}$, where $p(x) \in \ff\{x\}$, $ab \neq 0,k\in\N$. Then $(\C,A)$ is unirational. 
%\end{prop}
%\begin{proof}
%Denote $\alpha=(\frac{c-u^{(k)}}{b},\frac{u-p(({c-u^{(k)}})/{b})}{a}\big)$. Obviously $A $ is an irreducible differential polynomial and $A(\alpha)=0 $. Furthermore, there exists $B \in \ff\{x,y\}$ such that $\sat(B)= \mathbb I(\alpha)$. Since $\alpha$ is a proper parametrization of $(\C,B)$, by theorem \ref{th-order}, $\ord_yB=k$. Thus, $A \in \sat(B)$ and is divisible by $B$. This completes the proof.
%\end{proof}
% 

\section{Problems for further study}

There are several problems for further study. 
Given a linear differential rational parametrization $\mathcal P(u) =(\frac{P_1(u)}{Q_1(u)},\frac{P_2(u)}{Q_2(u)})$, 
Theorem \ref{th-propercriteria} {provides an algorithm to decide whether a given linear differential rational parametrization is proper or not, and in the affirmative case, to compute the implicit equation.} 
 It is  interesting to see whether in general the differential resultant can be used to compute the implicit equations of proper differential rational parametric equations.

The most important and unsolved problem is to give general methods to determine  the rational parametrizability of nonlinear differential curves and if so, to develop efficient algorithms to compute proper differential rational parametrizations. Motivated by the results for algebraic curves,
the determination problem may amount to define new differential invariants such as  differential genus for differential curves as proposed in \cite{FengGao} and \cite{Gao2003}.

\section{Acknowledgements}
We are grateful to  Franz Winkler, Xiao-Shan Gao, James Freitag, Ru-Yong Feng, Gleb Pogudin and Alexey Ovchinnikov  for helpful discussions, suggestions  and also encouragement when we work on unirational differential curves. This work is partially supported by NSFC Grants (11971029, 11688101, 11671014).


\section*{References}

\begin{thebibliography}{99}


\bibitem{Bronstein}M. Bronstein and M. Petkov\u{s}ek. An
introduction to pseudo-linear algebra. \textit{Theoretical
Computer Science}, 157(1):3-33, 1996.

  
 \bibitem{Ferro1997}
 G. Carr\`a-Ferro.
 A Resultant Theory for the Systems of Two Ordinary Algebraic Differential Equations.
 {\it Applicable Algebra in Engineering, Communication and Computing}, 8, 539-560, 1997.
 
 \bibitem{Ferro}
 G. Carr\`a-Ferro. A Resultant Theory for Ordinary Algebraic Differential
Equations. {\it Lecture Notes in Computer Science}, 1255, 55-65, Springer, 1997.
 
\bibitem{chardin}  M. Chardin. Differential Resultants and Subresultants. In: Proc. FCT’91. In: \textit{Lecture Notes in Computer Science}, vol. 529.
Springer-Verlag, 1991.

\bibitem{charset} T. Cluzeau and E. Hubert. Rosolvent Representation for Regular Differential Ideals. \textit{Appl. Algebra Engrg. Comm. Comput.}, 13, 395-425, 2003.


\bibitem{FengGao} R.Y. Feng and X.S. Gao. A polynomial time algorithm for finding rational general solutions of first order autonomous ODEs. \textit{Journal of Symbolic Computation}, 41(7):739-762, 2006.


%\bibitem{Gao1992} X.S. Gao, S.C. Chou. Implicitization of rational parametric equations. \textit{Journal of Symbolic Computation},14, 459-470,1992



 \bibitem{Gao2003}  X.S. Gao. Implicitization of differential rational parametric equations. \textit{Journal of Symbolic Computation}, 36(5), 811-824, 2003.
 
 \bibitem{Gao2013} X.S. Gao, W. Li, C.M. Yuan. Intersection theory in differential algebraic geometry: generic intersections and the differential Chow form. \textit{Trans. Am. Math. Soc}., 365 (9), 4575-4632, 2013.

\bibitem{Winkler2018} 
G. Grasegger, A. Lastra, J.R. Sendra, et al.. A solution method for autonomous first-order algebraic
partial differential equations, {\it J. Computational and Applied Mathematics}, 331, 88-103, 2018.


\bibitem{Hartshorne} R. Hartshorne. \textit{Algebraic geometry}. Springer-Verlag, New York, 1977. Graduate Texts in Mathematics, No. 52


\bibitem{Winkler2013} Y. Huang, L.X.C. Ng\^{o}, and F. Winkler. Rational General Solutions of Higher Order Algebraic ODEs. \textit{Journal of Systems Science and Complexity}, 26(2), 261-280, 2013.


%\bibitem{Igusa} J. Igusa. On a theorem of L\"uroth, \textit{Mem. Univ. Kyoto}, 26, 251-253,  1950-51.


\bibitem{Johnson1969}
J. Johnson.
Kahler Differentials and Differential Algebra.
\textit{Ann. Math.}, 89(1), 92-98,  1969.

%\bibitem{Kal}  M. Kalkbrener. Implicitization of Rational Curves and Surfaces. In Sakata, ed., \textit{Lect. Notes in Comp. Sci.} 508, AAECC-8, Tokyo, Japan,1990.


\bibitem{Kolchin1947} E.R. Kolchin. Extensions of Differential Fields, III. \textit{Bull. Amer. Math. Soc}., 53, 397-401, 1947.

\bibitem{Kolchin1973} E.R. Kolchin. {\it Differential Algebra and Algebraic Groups}. Academic Press, New York and London, 1973.

 \bibitem{Winkler2010} 
 L.X.C. Ng\^{o}  and F. Winkler. Rational general solutions of first order non-autonomous 
 parametrizable ODEs, {\it Journal of Symbolic Computation}, 45(12), 1426-1441, 2010.
 
 \bibitem{Winkler2011} 
 L.X.C. Ng\^{o} and  F. Winkler. Rational general solutions of planar rational systems of autonomous
ODEs, {\it Journal of Symbolic Computation}, 46(10), 1173-1186, 2011.

\bibitem{Ritt1932}
 J.F. Ritt.  {\it Differential Equations from the Algebraic Standpoint}.
 Amer. Math. Soc.,  New York, 1932.


\bibitem{Ritt}J.F. Ritt. {\it Differential Algebra}. Amer. Math. Soc., New York, 1950.

\bibitem{Rueda}  S.L. Rueda, J.R. Sendra.  Linear complete differential resultants and the implicitization of linear DPPEs. \textit{Journal of Symbolic Computation}, 45(3), 324-341, 2007.

\bibitem{Rueda2011} S.L. Rueda. A perturbed differential resultant based implicitization algorithm for linear DPPEs. \textit{Journal of Symbolic Computation}, 46, 977-996, 2011.

\bibitem{Sch98a} J. Schicho. Rational Parametrization of Surfaces. \textit{Journal of Symbolic
Computation}, 26, 1-9, 1998.

\bibitem{Winkler} J.R. Sendra, F. Winkler,  S. P\'erez-D\'iaz.
 \textit{Rational Algebraic Curves: A Computer Algebra Approach}. 
 Springer Publishing Company, Incorporated,  2007.

%\bibitem{SeW91} J. R. Sendra, F. Winkler. 
%Symbolic Parametrization of Curves. \textit{Journal of Symbolic Computation}; 12, 607–631 (1991).

\bibitem{Wsit}W.Y. Sit. The Ritt-Kolchin Theory for Differential Polynomials. In {\it Differential Algebra and
Related Topics}, 1-70, World Scientific, 2002.


%\bibitem{Movcur} T.W. Sederberg, R.N. Goldman, H. Du, lmplicitizing rational curves by the method of
%moving algebraic curves, {\it Journal of Symbolic Computation} 23, 153-175, 1997.

%\bibitem{Winkler2015} N.T. Vo and F. Winkler. Algebraic General Solutions of First Order Algebraic
%ODEs. In P. V. Gerdt, W. Koepf, M. W. Seiler, and V. E. Vorozhtsov, editors, {\it Computer Algebra in Scientific Computing}, volume 9301 of {\it Lecture Notes in Computer Science}, Cham, 2015. Springer International Publishing.


%\bibitem{Wal50} R.J. Walker. {\it Algebraic Curves}. Princeton Univ. Press (1950).

 \bibitem{Winkler2019}
F. Winkler. The Algebro-Geometric Method for Solving Algebraic Differential Equations --- A Survey.
 {\it Journal of Systems Science and Complexity}, 32(1), 256-270, 2019.

 




 \end{thebibliography}
\end{document}